\documentclass{article}
\usepackage{amsmath,amssymb,amsthm,amsfonts, amsthm}
\usepackage{graphicx}
\usepackage{bigints}
\usepackage{hyperref}
\hypersetup{colorlinks=true, linkcolor=blue,
    filecolor=magenta,  urlcolor=cyan
    }

\usepackage{xcolor}
\usepackage{comment}                                                                  

\newtheorem{theorem}{Theorem}

\newtheorem{corollary}[theorem]{Corollary}

\newtheorem{definition}[theorem]{Definition}

\newtheorem{lemma}[theorem]{Lemma}

\newtheorem{proposition}[theorem]{Proposition}
\newtheorem{remark}[theorem]{Remark}

\newcommand{\bR}{\mathbb{R}}
\newcommand{\bT}{\mathbb{T}}
\newcommand{\bN}{\mathbb{N}}
\newcommand{\bP}{\mathbb{P}}
\newcommand{\bE}{\mathbb{E}}

\newcommand{\cl}{\mathcal{L}}
\newcommand{\cL}{\mathcal{L}}
\newcommand{\cB}{\mathcal{B}}

\providecommand{\keywords}[1]{\textbf{\textit{Keywords}:} \quad #1}
\newcommand{\mscclass}[1]{\textbf{MSC2020:} \quad #1}

\usepackage[backend=biber, sorting=nyvt, maxbibnames=99]{biblatex}
\title{Extreme Value theory and Poisson statistics for discrete time samplings of stochastic differential equations.}
\author{ F. Flandoli\thanks{Scuola Normale Superiore. Email:franco.flandoli@sns.it } \and S. Galatolo\thanks{Dipartimento di Matematica - Università di Pisa. Email: stefano.galatolo@unipi.it orcid:0000-0003-3934-5412 }  \and
P. Giulietti\thanks{Dipartimento di Matematica - Università di Pisa. Email:paolo.giulietti@unipi.it orcid:0000-0001-9604-1699} \and S. Vaienti\thanks{Aix Marseille Universit\'e, Universit\'e de Toulon, CNRS, CPT, 13009 Marseille, France. Email: vaienti@cpt.univ-mrs.fr.}} 

\date{\today}

\begin{document}

\maketitle

\begin{abstract}
We investigate the distribution and clustering of extreme events of
stochastic processes constructed by sampling the solution of a Stochastic Differential Equation 
on $\bR^n$. We do so by studying the action of an annealead transfer operators on a suitable  spaces of 
densities.  The spectral properties of such operators are obtained by employing a mixture of techniques coming from SDE theory and a functional analytic approach to dynamical systems. 
\end{abstract}


\mscclass{Primary/Secondary 37A50, 37H05, 60G70, 60H50}

\tableofcontents

\section{Introduction} \label{sec:intro}

In the last decades modeling based on Stochastic Differential Equations has been used extensively in  different scenarios, in particular whenever 
the  statistical properties of extreme events are important, ranging from economy  (see e.g. \cite{1K}) to climate science (see e.g. \cite{1GL}). 
In this case the SDE model becomes a forecasting tool: one can investigate the likelihood of extreme events, as well as the typical behaviour of such systems. Here we focus on the former.

We consider a system whose evolution is described by a stochastic differential equation of the following type
\begin{equation*}  \left\{
\begin{aligned}
dX_{t}  &  =b\left(X_{t}\right)  dt+dW_{t}\\
X_{0}  &  =x, 
\end{aligned} \right.
\end{equation*}
where the domain of definition is $ \bR^n$ and $W_{t}$ is the Brownian motion defined on a probability space $\left(  \Omega,\mathcal{F},\mathbb{P}\right).$
It is well known that under suitable assumptions on the function $b$ (see Section \ref{sec:mainresults} for the details of our setup) 
the problem
above has a unique solution $(X_t^x)_{t \geq 0}$ and the resulting stochastic flow has a unique stationary measure $\mu$.

We consider a specific unbounded observable $g_{x_0}:\mathbb{R}^d\to \mathbb{R}$, encoding the appearance of an extreme event for a given point $x_0$ of the phase space, and sample the evolution of $g_{x_0}$ along the trajectories of   $X_t$ at discrete times.
This is done by choosing a time step $h\geq0$, a sequence of times $t_n=hn$ and the process $g_{x_0}(X^x_{t_n}) $. 

It is well known that pinpointing  a single extreme event is still, if possible at all, beyond the reach of current techniques, yet there are various possibilities to investigate the distribution, either temporal or spatial, of such events also in the sense of their clusterization. 
Here we consider the distribution of the extreme events for the process  $g_{x_0}(X^x_{t_n}) $; in particular we would like to find a sequence $\{ u_{n} \}_{n \in \bN}, u_n \in \mathbb{R^+}$, such that the following limit exists
\begin{equation} \label{problem}
\lim_{n\rightarrow\infty}\mathbb{P}\otimes \mu \left( \left \{ (\omega, x) \ : \  \max_{k=0,...,n-1}g_{x_0} \left(  X_{t_k}^x  
\right)  \leq u_{n}  \right \}   \right)  \in\left(  0,1\right) 
\end{equation}
where the $\otimes$ represents the product of measures.
The limit above encodes the notion of  extreme event, as long as $u_n \to \infty $ for the  process $g_{x_0}(X^x_{t_n} )$ where the initial conditions $x$  vary in the full set $\bR^d$ and are weighted 
according to $\mu$. 
Theorem \ref{thm:gen2} solves this problem i.e. it is possible to characterize such sequences and find an exponential distribution for the probability of not  exceeding the thresholds $u_n$, as it is usually done in the extreme values theory.

In the literature, such limit is often referred to as Extreme Value Law (EVL); its existence and its properties are strictly connected to the distribution of the hitting time of the process to a sequence of (properly renormalized) shrinking sets (see \cite{FFT10} for an introduction to such relation). 

Successively, we will  consider the distribution of extreme occurrences in a given time interval.   More precisely we will count
the number of visits of the time-discretized solution of our SDE to a shrinking sequence of  balls $B_n$ among the first $n$ steps of length $h,$ suitably normalized with the measure of the balls. We will show that the distribution of the number of visits will converge to a standard Poisson distribution  in the limit of large $n$, see Theorem \ref{thm:Poisson}.

We remark that for both results our construction is rooted in the functional approach of \cite{KL09, Keller12, AFGV0000}. 
The tools we use to get the results are related to the study of the functional analytic properties of the transfer operator $\cl_h$ associated to the SDE system  acting on suitable functional spaces.
In particular, we will relate the probability of occurrence of the rare events with the response of the dominating eigenvalue of $\cl_h$ to the  perturbation of the system constructed by adding a small hole corresponding to the rare event, transforming the initial system into an open one.

\subsection{Literature Review}

The study of extreme events has already been carried out in various framework and there is an extensive literature.   Comprehensive surveys can be found in the book \cite{RareBook16}, with a particular focus on dynamical systems and in  \cite{1LR} with a focus on stochastic processes.

Let us comment on some works which are closer to the present one which allow us also to enlighten  the main novelties of our contribution.

In the paper \cite{1D}, an Extreme Values Law in continuous time is found  in the case of one dimensional stochastic differential equations, using the properties of the Ornstein-Uhlenbeck process (see also \cite{1LR} for another point of view on these results).
 
For diffusion processes of gradient field type in $\bR^d$, the Ph.D thesis results of Kuntz \cite[Prop. 2.1, Theorem 2.4]{1ku}, prove exponential upper/lower bounds for $P ( M_T \leq R ) $, for large $R.$ Here $M_T=\max_{0\le t\le T}\|X_t\|,$ where $\{X_t\}_{t\ge 0}$ is a $\mathbb{R}^n$-valued stationary  reversible diffusion process and $\|\cdot\|$ the euclidean norm in $\mathbb{R}^n.$
About these continuous time results, we remark that, as we show in Appendix \ref{A3}, our explicit relation between  the time $T$ and the thresholds $R$ (see equation \eqref{5t} in Theorem \ref{thm:gen2})  used for the discrete time result is not the right one in the continuous case. 
The time/threshold rescaling used in the mentioned continuous time results seems to be quite implicit and not easy to  be determined in concrete examples.

In the absence of the drift term $b$,  the discretization of $ dX_t = dW_t $ is akin to a random walk. In the paper \cite{SPZw} the hitting time distribution for random walks in the line is described; these results are equivalent to EVL as proved in \cite{FFT10}.

Most of the preceding works obtained an extreme value distributions for continuous time processes beginning with specific SDEs, while in the present work we focus on discrete time but for a large class of Stochastic Differential Equations. The latter should verify assumptions {\bf H} and {\bf AD} in section \ref{sec:mainresults}, and also the conditions stated in Theorem \ref{MPZ}.  In this case we describe not only the distribution of the  ``first" rare event, but also the distribution of multiple extremes i.e. the Poisson statistics. We believe that our techniques could be generalized to more general SDE and  we will point out where these extensions could occur.

We also  reiterate the fact that our approach is directly based on the spectral properties of the transfer operators associated to the system and its response to suitable small perturbations representing rare events. 
This technique has been successfully applied to study random dynamical systems perturbed in an {\em annealed} way \cite{AFV15, CFVY}
 and the recent paper \cite{AFGV1} developed a spectral approach for a {\em quenched} extreme value theory that considers random dynamics and random observations.\\

 \section{Main Results} \label{sec:mainresults}

We consider a   stochastic differential equation on $\bR^d$ of the following type
\begin{equation} \label{eq:system1} \left\{
\begin{aligned}
dX_{t}  &  =b\left(X_{t}\right)  dt+dW_{t}\\
X_{0}  &  =x, 
\end{aligned} \right.
\end{equation}
 Here $W_{t}$ is the Brownian motion defined on a probability space $\left(  \Omega,\mathcal{F},\mathbb{P}\right)$. 
{We make the following assumptions on $b$:
\begin{description}
\item[H] (Lipschitz drift) There is $K$ such that $\forall x,y \in {\mathbb{R}^d}$
\begin{equation*}
|b(x)-b(y)|\leq K|x-y|.
\end{equation*}
\item[AD] (Dissipativity)  We assume that there exists  constants $R_1, R_2 \in \bR $  with $R_2>0$ such that for all $x \in \bR^d $
\begin{equation} \label{eq:dissiphyp}
\langle b(x), x \rangle \leq   R_1 - R_2 \|x \|^2 .
\end{equation}
\end{description}
}

\begin{remark}
The previous assumptions include the fundamental case 
\[
b(x)  =Ax+B(x)
\]
where $A$ is a matrix such that there exists $\nu>0$ with the property that
\begin{equation}
\left\langle Ax,x\right\rangle \leq-\nu\left\Vert x\right\Vert ^{2}%
\label{assumption on A}%
\end{equation}
for all $x\in\mathbb{R}^{d}$, and $B:\mathbb{R}^{d}\rightarrow\mathbb{R}^{d}$
is Lipschitz continuous and fulfills%
\begin{equation}
\langle B(x) , x \rangle =0\label{assumption on B}%
\end{equation}
or more generally
\begin{equation}
\left\langle B\left(  x\right)  ,x\right\rangle \leq C_{1}+C_{2}\left\Vert
x\right\Vert ^{2}\qquad\text{for all }x\in\mathbb{R}^{d}%
\label{assumption on C}%
\end{equation}
for some constants $C_{1}\in\mathbb{R}$ and $C_{2}<\nu$. Indeed assumption
\textbf{H} is obviously satisfied and assumption \textbf{AD} holds because%
\[
\langle Ax+B(x)  ,x \rangle \leq C_{1}-\left(
\nu-C_{2}\right)  \left\Vert x\right\Vert ^{2}.
\]
Condition \eqref{assumption on A}  is motivated for instance by the finite
dimensional discretization of Partial Differential Equations (PDE henceforth)
of parabolic type, where $A$ is the discretization of the Laplacian or, more
generally, of the second order strongly elliptic part. Condition
\eqref{assumption on B} is motivated by the discretization of nonlinear
operators like the inertial (convective) operator of the Navier-Stokes
equations;  we impose the Lipschitz continuity on $B$ for simplicity.\footnote{This is not satisfied by quadratic operators, therefore an additional Lipschitz
cut-off is necessary to get a finite dimensional operator as above.} The
generalization of Condition \eqref{assumption on C} may help to accommodate
linear first order operators in the discretization of a PDE.

\end{remark}

Let $x_{0}\in \bR^d$ be a chosen point of the phase space. Let  $g_{x_0}: \bR^d \setminus \{x_0 \} \to \bR $ be
\begin{equation} \label{eq:distance1}
g_{x_0} (x) :=-\log d( x,x_{0})
\end{equation}
where $ d$ is the euclidean distance. 
The observable $g_{x_0}$ hence measures how far we are from our chosen $x_0$ on a logarithmic scale. 
Let $h > 0$ and  let  the sequence of  discrete times in which we sample the process be denoted as
\begin{equation}\label{TS}
t_{k}:=kh.
\end{equation}

Given the assumptions \textbf{H} and \textbf{AD}, the system \eqref{eq:system1} has a unique invariant measure $\mu$ (see Section \ref{sec:Rdperturbation}  for more details).

We now consider the problem \eqref{problem}, and  for the reader convenience we state below the first main result of our work.

 \begin{theorem} \label{thm:gen2}
 Let  $h,\tau>0$ and let $X_t$ be the solution of \eqref{eq:system1} at time $t$.
 Let $u_n$ be a real sequence such that 
 \begin{equation}\label{5t}
 n\, \mu (B_n) \to \tau
 \end{equation}
 where $B_{n}$ is the ball $B\left(  x_{0},e^{ -u_{n}}  \right)$. Let $ t_{k}:=kh$. Then
\begin{equation}\label{th2}
\lim_{n\rightarrow\infty}\mathbb{P} \otimes \mu \left( \left\{ (\omega  , x) \ : \  \max_{k=0,...,n-1}g_{x_0} \left(  X_{t_k}^x  
\right)  \leq u_{n}  \right \} \right) = e^{-\tau}.
\end{equation}
 \end{theorem}

 \begin{remark}
Analogous results holds if the system \eqref{eq:system1} is defined on $\bT^d$. As many of the constructions presented here are either not necessary or can be simplified in such case, the explicit discussion of this case will appear somewhere else. 
\end{remark}
\begin{remark}
We remark that the asymptotic law found in  \eqref{th2} does not depend on $h$. 
This is in our opinion due to the presence of a uniform noise, analogously to similar results for dynamical systems perturbed by annealed noise (see \cite{AFV15}).
In the Appendix \ref{A4} we show an example where the lack of sufficient noise gives a different asymptotic law.
\end{remark}

The second main result we prove is a refinement of the first one and is   about the distribution of multiple occurrences of the extreme events.
Let $B_n$ be again a ball centered where needed and of radius $e^{-u_n}.$ We are now interested in studying the distribution of the number of visits to the set $B_n$ in a prescribed time interval. We follow the strategy recently used in \cite{AFGV0000} for this kind of results.   
We argue  that the exponential law given by the extreme value distribution
describes the time between successive events in a Poisson process.  We begin by introducing the following random variable
$S_{n}:=\sum_{i=0}^{n-1}1_{B_n}(X^x_{ih}),$ 
which counts the number of visits of the time sampled solution to the ball $B_n$  among the first $n$ iterations of the process. In order to get a limiting distribution when $n\rightarrow \infty$ we have to rescale time as we did in (\ref{5t}). Accordingly
\begin{definition}\label{8t}
Let us take $\tau>0$ and $n\ge1.$ We  define the sequence of discrete random variables
\begin{equation}\label{6t}
S_{n,\tau}:=\sum_{i=0}^{\lfloor\frac{\tau}{\mu(B_n)}\rfloor}1_{B_n}(X^x_{ih}).
\end{equation}
We say that $S_{n,\tau}$ converge in distribution to the discrete random variable $W,$ possibly  defined on a different probability space and with distribution $\nu_W,$ if we have for any $k\in \mathbb{N}:$
\begin{equation}\label{7t}
\lim_{n\rightarrow \infty}\mathbb{P}\otimes\mu\left(\sum_{i=0}^{\lfloor\frac{\tau}{\mu(B_n)}\rfloor}1_{B_n}(X^x_{ih})=k\right)=\nu_W(\{k\}).
\end{equation}
\end{definition}
The second main result of this work is the following theorem, establishing the convergence toward a standard Poisson distribution.
\begin{theorem}\label{thm:Poisson}
 Let  $\tau>0$, $X_t$ be the solution of \eqref{eq:system1} on $\bR^d $,  let $B_{n}$ the ball $B\left(  x_{0},\exp\left(  -u_{n}\right)  \right)$ and let $u_n$ such that $ n \, \mu (B_n) \to \tau $ . Then
 $$
\lim_{n\rightarrow \infty}\mathbb{P}\otimes\mu\left(\sum_{i=0}^{\lfloor\frac{\tau}{\mu(B_n)}\rfloor}1_{B_n}(X^x_{ih})=k\right)=\frac{e^{-\tau}\tau^k}{k!}.
 $$
    \end{theorem}

As mentioned before, the proofs of Theorem \ref{thm:gen2} and Theorem  \ref{thm:Poisson} rely on the spectral perturbation  results of  \cite{KL09, Keller12}.
We begin by defining the transfer operator $\cl_h$ associated to the evolution of the system for a fixed time $h>0$.
We find Banach spaces where $\cl_h$ satisfies a Lasota-Yorke type inequality which in turns, if the spaces also embed compactly, implies a spectral gap. A key inequality will be obtained by exploiting the regularizing effect of the noise in the stochastic differential equation and the presence of the dissipative assumption (AD). 
Since the ambient space is $\mathbb{R}^d$ and therefore non compact  and we use singular perturbations generating discontinuous densities,  we introduce $BV_{\alpha}$ as the spaces of bounded variation densities decaying at infinity with a certain power law of exponent $\alpha$.
We remark that the use of bounded variation spaces, rather then the more regular Sobolev Space, has already been very fruitfull in the study of SDE (see \cite{fuku2000, AMMP2010} to begin with). 

Once collected all the necessary estimates, the proofs of Theorem \ref{thm:gen2} and Theorem \ref{thm:Poisson} are spelled out respectively in Section
\ref{app:repfo2} and Section \ref{sec:Poisson}.

\subsection{Organization of the paper} 
The plan of the paper is as follows:
in Section \ref{S3} we set up the functional analytic framework necessary to study our problems and we perform the regularization estimates.

 Section \ref{app:repfo} recalls the main abstract tool we use, namely the asymptotic behavior of the largest perturbed eigenvalue stated in  Proposition \ref{thm:repfo}.
 Section \ref{app:repfo2} shows how from Proposition \ref{thm:repfo} we can recover the extreme event laws  stated in Theorem \ref{thm:gen2}.

Section \ref{sec:Poisson} shows how from Proposition \ref{thm:repfo} we can recover the Poisson statistics  stated in Theorem \ref{thm:Poisson}.

Appendix \ref{A3} contains remarks on the choice of the time discretization step $h$ and how, sending $ h \to 0$, does not give further information on the EVL. 
 Appendix \ref{A4} contains an example of a different system and a different asymptotic limit (with respect to our main theorem) which shows some heuristic about the effect of local fluctuations on the limit law. 
 
\section{Transfer operators, Banach spaces and regularization lemmas.} \label{S3}
We first define and study the basic properties of the transfer operators associated to our system in sections \ref{sec:operators} and \ref{sec:PO}. 

In section \ref{sec:spacesRD}, we introduce a weighted $L^1$ space and bounded variation spaces which we will use to carry out the necessary estimates.

In sections \ref{sec:spacerdtotd} and \ref{rpo} we study the  regularization properties of the transfer operators when applied to these suitable functional spaces.

 Recall the SDE  \eqref{eq:system1} on  the euclidean space $\mathbb{R}^d$, assuming the regularity assumption $(H)$ and the dissipative assumption $(AD)$. 
For this kind of stochastic differential equations it is known (see
\cite{Friedman64}, \cite{SV}) that the equation has the properties
of strong existence of the solutions and pathwise uniqueness.

Moreover,  Menozzi, Pesce and Zhang \cite{MPZ} prove bounds on the transition probabilities for these systems (so called Aronson type estimate) in this setup. Such estimates imply that the transfer operator associated to the system  has a regular kernel, and hence regularizing properties.
Let $\theta _{t}$, for $t\geq 0$ be the flow solving 
\begin{equation*}
\left\{ 
\begin{array}{l}
\dot{\theta}_{t}(x)=b(\theta_{t}(x)) \\ 
\theta _{0}(x)=x.%
\end{array}%
\right. 
\end{equation*}
for  a function $b$ of the system considered.  Let $\lambda \in (0,1], t>0$ and $g_{\lambda}$ be the Gaussian
distribution 
\begin{equation*}
g_{\lambda}(t,x):=t^{-\frac{d}{2}}e^{\frac{-\lambda |x|^{2}}{t}}.
\end{equation*}

\begin{theorem}[\protect\cite{MPZ}, Theorem 1.2 and Remark 1.3.]\label{MPZ}
Let us fix  $T>0$. For each $0<t\leq T$ and  $x\in \mathbb{R}^{d}$\ let  $X_{t}(x)$ be the unique  solution of \eqref{eq:system1} starting from $x$ at time $t $. \ Then   $X_{t}(x) $ has a density  for each $y\in \mathbb{R}^{d}$ that can be expressed as a function $
S_t(x, y)$  which is continuous in both variables. Moreover $S_t$ satisfies the following:

\begin{description}
\item[1] (Two sided density bounds) There exist $\lambda _{0}\in
(0,1],C_{0}\geq 1$ depending on $T,k,d$ such that for any $x,y\in \mathbb{R}%
^{d},~t<T$%
\begin{equation*}
C_{0}^{-1}g_{\lambda _{0}^{-1}}(t,\theta _{t}(x)-y)\leq S_t(x,y)\leq
C_{0}g_{\lambda _{0}}(t,\theta _{t}(x)-y).
\end{equation*}

\item[2] (Gradient estimates) There exist $\lambda _{1}\in (0,1], C_{1}\geq 1$ depending on $T,k,d$ such that for any $x,y\in \mathbb{R}^{d},~t<T$%
\begin{equation*}
|\nabla _{x}S_t(x,y)|\leq C_{1}t^{-\frac{1}{2}}g_{\lambda _{1}}(t,\theta
_{t}(x)-y),
\end{equation*}
\begin{equation*}
|\nabla _{y}S_t(x,y)|\leq C_{1}t^{-\frac{1}{2}}g_{\lambda _{1}}(t,\theta
_{t}(x)-y).
\end{equation*}
\end{description}
\end{theorem}
\subsection{The Kolmogorov operator and the  transfer  operator}
\label{sec:operators}
In this section we define the transfer operators associated to the evolution of a SDE  and show some of the basic properties of these operators. The properties of the transition probabilities $S_t$ inherited by \cite{MPZ} 
will be used  to define a Kolmogorov  operator (composition operator) associated to our system.
\begin{definition} \label{def:koldef}
The Kolmogorov operator $P_{t}:L^\infty(\mathbb{R}^d)\to C^0(\mathbb{R}^d)$ associated to the system \eqref{eq:system1} is defined  as follows. Let $\phi \in L^{\infty}(\mathbb{R}^d)$ then $\forall x \in \mathbb{R}^d $
\begin{equation*}
(P_{t}\phi )(x):=\mathbb{E}[\phi (X_{t}(x))].
\end{equation*}
\end{definition}
In the literature this is also known as  stochastic Koopman operator. Given the transition probabilities $S_t$ it holds 
\begin{equation*}
(P_{t}\phi )(x)=\int_{\bR^d} \phi (y)S_t(x,y)dy.
\end{equation*}
By this we see that 
\begin{equation}\label{inft}
\|P_{t}\phi\|_\infty\leq \|\phi \|_\infty.
\end{equation}

Now we define the transfer operator $\mathcal{L}_t:L^1(\mathbb{R}^d)\to L^1(\mathbb{R}^d)$ by duality. 
If $\nu $ is a Borel signed measure on $\mathbb{R}^{d}$
\begin{equation*}
\int_{\bR^d} (P_{t}\phi )(x)d\nu (x)=\int_{\bR^d} \int_{\bR^d} \phi (y)S_t(x,y)dyd\nu (x)
\end{equation*}%
supposing that  $\nu $ has a density with respect to the Lebesgue
measure 
$f\in L^1(\mathbb{R}^d)$ i.e $ d\nu = f(x) dx $
we can thus write 
\begin{equation}\label{duality}
\int_{\bR^d} (P_{t}\phi )(x)d\nu (x) = \int_{\bR^d} \phi (y) \left( \int_{\bR^d} S_t(x,y)f(x)dx \right) dy.
\end{equation}
We can then define the transfer operator associated to the system and to the
evolution time $t$ as
\begin{definition}[transfer operator] Given $f\in L^1(\mathbb{R}^d)$ we define the measurable function $\cl_{t}f:\mathbb{R}^d\to \mathbb{R}^d$ as follows. For almost each $y\in \mathbb{R}^d$ let 
\begin{equation}\label{eq:stoctransfer}
(\cl_{t}f)(y):=\int S_t(x,y)f(x)dx .
\end{equation}

\end{definition}
By \eqref{duality}  we recover the duality relation between the Kolmogorov and the transfer operator whenever $d\nu = f(x) dx $ i.e.
\begin{equation}\label{eq:stocduality}
\int (P_{t}\phi )(x)d\nu (x)=\int \phi (y)(\cl_{t}f)(y)dy.
\end{equation}

\begin{lemma}The operator $\mathcal{L}_t$ preserves the integral and is a weak contraction with respect to the $L^1$ norm.\label{7}
\end{lemma}
\begin{proof}
The statements follow directly from \eqref{inft} and 
\eqref{eq:stocduality}  by setting $\phi=1$.
\end{proof}
Since clearly $\mathcal{L}_t$ is a positive operator, we also get that $\mathcal{L}_t$ is a Markov operator having kernel $S_t$.
\begin{remark}(\text{On the stationary measure}).

In the following we will define suitable spaces where the operators $\mathcal{L}_t$ have nice spectral properties, allowing the study of their leading eigenvalues. As a byproduct of this analysis we will give a proof of the existence and uniqueness of the stationary measure $\mu.$ We remind that the stationary measure satisfies
$$
\int \phi(x) d\mu(x)=\int (P_{t}\phi )(x)d\mu (x).
$$
Whenever it is absolutely continuous, its density $f_0$ is the fixed point of the transfer operator: $\mathcal{L}_tf_0=f_0.$ The existence and uniqueness of such a fixed point will  follow directly from the proof of {\bf Assumption R1} in section \ref{sec:estimatesper}.
\end{remark}
\subsection{Perturbed operators}
\label{sec:PO}
To prove the extreme event laws shown in Theorem \ref{thm:gen2}  
we will use the construction outlined in \cite{KL09, Keller12} 
(see Section \ref{app:repfo}) adapted to our case.

This is based on the idea of considering the target set $B_n$ as a hole in the phase space and consider it  as an open system.
Then we will study it by means of the related transfer operators and their properties.

\begin{definition}
Let $ t >0$, $ x\in \mathbb{R}^d $ 
and  $u_n \to 0$ be a real sequence.  We denote by $B_{n}$ the ball $B\left(  x,\exp\left(  -u_{n}\right)  \right)$.
We define the ``perturbed"  versions of the Kolmogorov and transfer operators by setting
\begin{equation} \label{def:twistedoperator}
\begin{split} (P_{t,n} \phi)(x):=\mathbb{E}[1_{B^c_n}(X^x_t)\phi(X^x_t)], \ \phi\in L^\infty(\bR^d)  \\
(\cl_{t,n}f)(x):=1_{B^c_n}(x)(\mathcal{L}_t f)(x), \ f\in L^1(\bR^d).
\end{split}
\end{equation}
\end{definition}

Analogously to what is done in Section \ref{sec:operators} one can prove that the perturbed operators  enjoy the following  duality relation:  
\begin{equation} \label{mod} 
\begin{split}
\int (P_{t,n}\phi)(x) f(x) dx \hspace{-3pt}  & =
\iint \phi(y) 1_{B_n^c}(y)S_t(x,y) f(x) dydx \\
& =\int \phi(y) 1_{B_n^c}(y) \int S_t(x,y) f(x) dx dy \\
&= \int \phi(y)\left[1_{B_n^c}(y)(\mathcal{L}_tf)(y)\right]dy=\int \phi(y)(\cl_{t,n} f)(y)dy.
\end{split}
\end{equation}

 The iterates of the perturbed operator inherit their properties from the following lemma.  
\begin{lemma}\label{ror}
For every $t,s\geq0$,  $\phi \in L^\infty\left( \bR^d \right)$   
\[
P_{t,n}\left(  P_{s,n}\left(
\phi\right)  \right)  \left(  x\right)  =\mathbb{E}\left[  1_{B_{n}^{c}%
}\left(  X_{t}^{x}\right)  1_{B_{n}^{c}}\left(  X_{t+s}^{x}\right)
\phi\left(  X_{t+s}^{x}\right)  \right].
\]
\end{lemma}

\begin{proof}%
\begin{align*}
P_{t,n}\left(  P_{s,n}\left(
\phi\right)  \right)  \left(  x\right)   &  =\mathbb{E}\left[  1_{B_{n}^{c}%
}\left(  X_{t}^{x}\right)  P_{s,n}  \left(  \phi\right)
\left(  X_{t}^{x}\right)  \right] \\
&  =\mathbb{E}\left[  1_{B_{n}^{c}}\left(  X_{t}^{x}\right)  \mathbb{E}\left[
1_{B_{n}^{c}}\left(  X_{s}^{y}\right)  \phi\left(  X_{s}^{y}\right)  \right]
_{y=X_{t}^{x}}\right] \\
&  =\mathbb{E}\left[  1_{B_{n}^{c}}\left(  X_{t}^{x}\right)  \mathbb{E}\left[
1_{B_{n}^{c}}\left(  X_{t+s}^{x}\right)  \phi\left(  X_{t+s}^{x}\right)
|\mathcal{F}_{t}\right]  \right]
\end{align*}
(by Markov property and where $\mathcal{F}_{t}$ is the filtration adapted to $W_t$)%
\begin{align*}
&  =\mathbb{E}\left[  \mathbb{E}\left[  1_{B_{n}^{c}}\left(  X_{t}^{x}\right)
1_{B_{n}^{c}}\left(  X_{t+s}^{x}\right)  \phi\left(  X_{t+s}^{x}\right)
|\mathcal{F}_{t}\right]  \right] \\
&  =\mathbb{E}\left[  1_{B_{n}^{c}}\left(  X_{t}^{x}\right)  1_{B_{n}^{c}%
}\left(  X_{t+s}^{x}\right)  \phi\left(  X_{t+s}^{x}\right)  \right]
\end{align*}
by the basic properties of the conditional expectation (see \cite[Section 9.7]{Williams}).
\end{proof}

\begin{remark}
In particular, given generic $s,t \in \mathbb{R}^+$
\[
P_{t+s,n}\left(  \phi\right)  \left(  x\right)  \neq
P_{t,n}\left(  P_{s,n}\left(
\phi\right)  \right)  \left(  x\right)
\]
namely $P_{t,n}$ is not a semigroup. However notice that if $t_{1}<t_{2}$, then
(we take $t=t_{1}$ and $t+s=t_{2}$ above)%
\[
P_{t_{1},n}\left(  P_{t_{2}-t_{1},n}
\left(  \phi\right)  \right)  \left(  x\right)  =\mathbb{E}\left[
1_{B_{n}^{c}}\left(  X_{t_{1}}^{x}\right)  1_{B_{n}^{c}}\left(  X_{t_{2}}%
^{x}\right)  \phi\left(  X_{t_{2}}^{x}\right)  \right]  .
\]
\end{remark}

Thus, for  $P_t^{(n)}$ we have the following
\begin{corollary} \label{cor:EtoTwistedP}
For every $x\in \bR^d$, $0=t_{0}<t_{1}<...<t_{n}$, $\phi\in
L^\infty \left( \bR^d \right)  $, we have
\[
\left(  P_{t_{0},n}\circ P_{t_{1}-t_{0},n}\circ\cdot\cdot\cdot\circ P_{t_{n}-t_{n-1},n}\right)
\left(  \phi\right)  \left(  x\right)  =\mathbb{E}\left[  1_{B_{n}^{c}}\left(
X_{t_{0}}^{x}\right)  \cdot\cdot\cdot1_{B_{n}^{c}}\left(  X_{t_{n}}%
^{x}\right)  \phi\left(  X_{t_{n}}^{x}\right)  \right].
\]
\end{corollary}

{\bf Notations} For general properties of the unperturbed and perturbed transfer operators, we will use respectively  the notations $\cl_t$ and $\cl_{t,n}$ and we will switch from $t$ to $h$ when needed. We will denote indifferently with $L^p$ of $L^p(\mathbb{R}^n), p\ge 1$ the space of complex-valued measurable  functions which are $p$-summable with respect to the Lebesgue measure on $\mathbb{R}^n.$
 
\subsection{Functional Spaces: quasi-H\"older space  on $\bR^d$ } 

\label{sec:spacesRD}

We now define suitable functional spaces on which the transfer operators introduced in the previous sections have a regularizing behavior and nice spectral properties.
We construct  spaces which are suitable for our non-compact environment, imposing a controlled behavior far away from the origin  by using  weight functions growing at infinity.
These spaces will be denoted as  $BV_{\alpha}$ and $ L^1_{\alpha} $.
Let $ \alpha > 0 $ and define the weight function
\begin{equation}\label{WF}
\rho_{\alpha} \left(  \left\vert x\right\vert \right)  =\left(  1+\left\vert
x\right\vert ^{2}\right)  ^{\alpha/2}.
\end{equation}
Let $L_{\alpha}^{1}\left(  \mathbb{R}^{d}\right)  $ be the 
space of Lebesgue measurable $f:\mathbb{R}^{d}\rightarrow\mathbb{R}$ 
such that
\[
\left\Vert f\right\Vert _{L_{\alpha}^{1}\left(  \mathbb{R}^{d}\right)  }%
:=\int_{\mathbb{R}^{d}}\rho_{\alpha} \left(  \left\vert x\right\vert \right)  \left\vert
f\left(  x\right)  \right\vert dx<\infty.
\]
Note that,  $ L^1_{\alpha} \subset L^{1} $ and if $ f \in L^1_{\alpha} $ then $ \| f \|_{L^1} \leq \| f \|_{L^1_{\alpha}} $. Moreover for $\alpha= 0$, $ L^1_0 = L^1$.

We now adapt the Bounded Variation spaces used in \cite{saussol00}, see also \cite{KGE}, to the setup at hand.\footnote{These kind of spaces have also been used already in the context of extreme value theory (see \cite{FGGV18, CFVY19}).} For a Borel subset $S\subseteq \mathbb{R}^{d}$ let us define 
\begin{equation*}
osc(f,S)=
 \operatorname*{ess\,sup}_{x\in S}f
-\operatorname*{ess\,inf}_{x\in S}f.
\end{equation*}
We now define the seminorm:\footnote{In the definition of \cite{saussol00} and \cite{KGE} we have that  $ \epsilon < \epsilon_0$,  for  a suitable $\epsilon_0>0.$ However in our case is not restrictive to choose  $\epsilon_0=1$. By inspecting carefully the proof of Lemma \ref{imm} and Proposition \ref{prop:verR3_Rd} we see that our choice of $\epsilon_0 = 1 $ in the definition of oscillation does not impose any limitation. 
By Lemma \ref{lem:Rdregularization} we see that, beside the discontinuities created by our $1_{B_n}$, the  oscillation of $\mathcal{L}_t f $ can be easily controlled, as $\mathcal{L}_t f $ is regularized to $C^1$. This setup is different from those originally in \cite{saussol00}, where  $\epsilon_0$ was tied to the size of the partitions and the expansions of the considered maps, and had to be chosen accordingly. 
}
\begin{equation*}
\|f\|_{osc(\bR^d)}=\sup_{\epsilon \leq 1 }\epsilon ^{-1}\int_{\bR^d} osc(f,B_{\epsilon }(x)) d\psi(x),
\end{equation*} 

Here the measure $\psi$ is a Radon probability measure on $\mathbb{R}^d$ and we require that:
\begin{itemize}
    \item [($A_{\psi}1$)] $\psi$ is absolutely continuous with respect to the Lebesgue measure, having a continuous bounded density $\psi'$ such that $\psi'>0$  everywhere. 
    \end{itemize}

We can define a $\| \cdot \|_{BV_{\alpha}}$ norm by setting
\begin{equation} \label{def:BValpha}
\left\Vert f\right\Vert _{BV_{\alpha}}:=\left\Vert f\right\Vert _{L_{\alpha
}^{1}\left(  \mathbb{R}^{d}\right)  }+\sup_{\epsilon\in\left(  0,1\right]
}\epsilon^{-1}\int_{\mathbb{R}^{d}}osc\left(  f,B_{\epsilon}\left(  x\right)
\right)  d\psi(x).
\end{equation}
It is not difficult to show that $ \| \cdot \|_{BV_{\alpha}}$ indeed defines a norm and the set of $L^1_{\alpha}$ functions for which this norm is finite is a Banach space which we denote by  $BV_{\alpha}\left(  \mathbb{R}^{d}\right) $.\footnote{ The proof can be obtained by adapting Propositions B.4 and B.5 in B. Saussol PhD thesis  (\cite{STh}). First of all one notices that 
 $L^1_{\alpha}(\mathbb{R}^{d})$ is complete. Then, if  $f_n$ is a Cauchy sequence in  $BV_{\alpha},$ it is also Cauchy in  $L^1_{\alpha}(\mathbb{R}^{d})$ and therefore in $L^1(\psi),$ since $\psi'\in L^{\infty}(\mathbb{R}^{d}).$ Then one finally applies Propositions B.4 and B.5 in Saussol's thesis which explicitly uses $L^1(\psi)$ for the oscillatory part.}

We prove that our weighted bounded variation space is compactly immersed in $L^1.$
\begin{theorem} \label{thm:embeddingRD}
$BV_{\alpha}\left(  \mathbb{R}^{d}\right)  \hookrightarrow L^{1}\left(  \mathbb{R}%
^{d}\right)$ is a compact embedding. 
\end{theorem}

\begin{proof}
We prove that given a sequence $g_{n}\in L^{1}$ such that $%
\|g_{n}\|_{BV_{\alpha }}\leq M$ for some $M$ there is a subsequence $%
g_{n_{k}}$ and $g\in L^{1}$ such that $\|g_{n_{k}}-g\|_{L^{1}}\rightarrow 0$.

Let us consider a sequence $B_{m}=B_m(0)$ of balls centered in the origin with radius $m$, eventually covering $\mathbb{R}^{d}$. Let us fix $m$. By the
fact that on a compact domain the usual $BV$ topology is equivalent to $BV_{\alpha}$ and the space $BV(B_{m})$  has a compact immersion
in $L^{1}(B_{m})$ there is a subsequence $g_{n_{m,k}}$ and a function $
f_{m}:B_{m}\rightarrow \mathbb{R}$ such that $g_{n_{m,k}}$ restricted to $B_{m}$ converges to $f_{m}$ \ in the $L^{1}$ topology.

Let us define the extension $\overline{f_{m}}$ of $f_{m}$ to $\mathbb{R}^{d}$ by 
\begin{equation*}
\overline{f_{m}}(x)=\left\{ 
\begin{array}{c}
f_{m}(x)~if~x\in B_{m} \\ 
0~if~x\notin B_{m}.%
\end{array}%
\right. 
\end{equation*}
Since $\|g_{n_{m,k}}\|_{L^{1}}\leq M$ \ we also have $\|\overline{f_{m}}(x)\|_{L^{1}}\leq M$. 
Once found $g_{n_{m,k}}$ and $f_m$, we then consider $B_{m+1}$ and from the sequence $g_{n_{m,k}}$ let us draw a subsequence $g_{n_{m+1,k}}$ converging on $B_{m+1}$ to some $f_{m+1}$. Being a subsequence of the previously extracted sequence, $g_{n_{m+1,k}}$ will converge to $f_m$ on $B_m$ and then $f_m=f_{m+1}$ on $B_m$.
We can then continue inductively and define for each $m\geq 0$ a subsequence $g_{n_{m,k}}$ and a function $f_m$ with $g_{n_{m,k}}\to f_m$ on $B_m$. Furthermore we will also have an extension $ \overline{f_{m}}  $ on $\mathbb{R}^{d}$ for each $m\geq0$.
Thus the sequence $m\to \overline{f_{m}}$ \
converges pointwisely to some function $f:\mathbb{R}^{d}\rightarrow \mathbb{R%
}.$ The sequence $\left\vert \overline{f_{m}}(x)\right\vert $ \
 is an increasing sequence and then by the monotone convergence theorem 
$f\in L^{1}$ and $||f||_{L^{1}}\leq M$.

Now for each $m$ \ consider $k_{m}$ such that 
\begin{equation*}
\int_{B_{m}}|g_{n_{m,k_{m}}}-f_{m}|dx\leq \frac{1}{m}.
\end{equation*}

Since $\|g_{n_{m,k_{m}}}\|_{ L_{\alpha }^{1}}$ is uniformly bounded we have that there is some $M_2>0$  independent of $m$ 
 such that $\int_{\mathbb{R}%
^{d} \setminus B_{m}}|g_{n_{m,k_{m}}}|dm\leq \frac{M_2}{\rho_{\alpha}(m)}.$
We have that 
\begin{equation*}
||g_{n_{m,k_{m}}}-\overline{f_{m}}||_{L^{1}}\leq \frac{1}{m}+\frac{M_2}{\rho_{\alpha}(m)}
\end{equation*}%
and thus $g_{n_{m,k_{m}}}\rightarrow f$ in the $L^{1}$
topology.
\end{proof}

Let us suppose now  that $f\in BV_{\alpha}$ and let $\mathcal{K}$ be a compact set in $\mathbb{R}^d.$ We will need later on an  estimate of the $L^{\infty}(\psi)$ norm of $f$ on $\mathcal{K}$, designated as $||f||_{L^{\infty}(\psi, \mathcal{K})}.$ 
\begin{proposition}\label{imm} 
For any compact set $\mathcal{K}\subset \mathbb{R}^d$ we have
 
\begin{equation}\label{in}
\|f\|_{L^{\infty}(\psi, \mathcal{K})}\le \frac{\max\left(||\psi'||_{L^{\infty}(\mathbb{R}^d)},1\right)}{d_{\mathcal{K}}}||f||_{BV_{\alpha}},
\end{equation}
where again $\psi'$ denotes the density of $\psi$ with respect to the Lebesgue measure and  $d_{\mathcal{K}}=\text{essinf}_{x\in \mathcal{K}}\psi(B_1(x))$. (We remark that by $A_{\psi}1$ we have $d_{\mathcal{K}} >0$.)
\end{proposition}

   Following  Proposition B.6 in \cite{STh} we can write for any $x\in \mathcal{K}$ and $y\in B_1(x):$
   $|f(x)|\le |f(y)|+\text{osc}(f, B_1(x)).$ 
By integrating in $y$ over $B_1(x)$ we get 
\begin{equation}\label{qui}
  \| |f(\cdot)|\psi(B_1(\cdot)) \|_{L^{\infty}(\psi, \mathcal{K})}\le ||f||_{L^1(\psi)}+\epsilon||f||_{osc(\mathbb{R}^d)}\le ||f||_{L^1(\psi)}+||f||_{osc(\mathbb{R}^d)}.
\end{equation}

By using  the fact that $||f||_{L^1(\psi)}\le ||\psi'||_{L^{\infty}(\mathbb{R}^d)} ||f||_{L^1_2(\mathbb{R}^d)},$ we finally get the statement.

\begin{remark}
  From now on we will denote with   $||f||_{L^{\infty}( B)}$ the $L^{\infty}$ norm with respect to the Lebesgue measure restricted to the compact set $B.$ By the assumption ($A_{\psi}1$), $||f||_{L^{\infty}( B)}$ differs from $L^{\infty}(\psi, B)$ by a multiplicative constant depending only upon $B$. 
\end{remark}

\subsection{Regularization properties for the transfer operator.}
\label{sec:spacerdtotd}

In this section we see how the properties of the SDE  \eqref{eq:system1} we consider, such those derived from the Brownian motion, have a regularizing effect at the level of the associated transfer operators.

In the following lemma the notation $L_{2}^{1}$ stands for the space $L_{\alpha}^{1}$ when $\alpha=2$.

\begin{lemma}
\label{lem:pbound} Given $h>0$, there exist constants $A,B>0$ and $%
\lambda\in\left( 0,1\right) $ (also depending on $h$) such that 
\begin{equation}
\left\Vert\cl_{h}^{n}f\right\Vert _{L_{2}^{1}}\leq A\lambda^{n}\left\Vert
f\right\Vert _{L_{2}^{1}}+B\left\Vert f\right\Vert _{L^{1}} 
\label{ineq1 lemma 18}
\end{equation}
for every $f\in L_{2}^{1}$ and every $n\in\mathbb{N}$.
\end{lemma}

\begin{proof}
\textbf{Step 1}. Call $L_{\text{dens}}^{1}$ the set of all probability
density functions, namely the elements $p\in L^{1}$ such that $p\geq0$ a.s.
and $\left\Vert p\right\Vert _{L^{1}}=1$.

The statement of the lemma is equivalent to prove there exist constants $%
C,D>0$ and $\lambda\in\left( 0,1\right) $ such that inequality%
\begin{equation}
\left\Vert \cl_{h}^{n}p\right\Vert _{L_{2}^{1}}\leq C\lambda^{n}\left\Vert
p\right\Vert _{L_{2}^{1}}+D   \label{ineq2 lemma 18}
\end{equation}
holds true for every $p\in L_{\text{dens}}^{1}\cap L_{2}^{1}$. That the
statement of the lemma implies this new one is obvious (with the same
constants). Let us prove the converse. Take $f\in L_{2}^{1}$ and call $%
f^{+}=f\vee0$, $f^{-}=\left( -f\right) \vee0$ (thus $f=f^{+}-f^{-}$). Since $%
\left\vert f^{\pm}\left( x\right) \right\vert \leq\left\vert f\left(
x\right) \right\vert $, we have $f^{\pm}\in L_{2}^{1}\subset L^{1}$. Call $%
Z_{\pm}=\left\Vert f^{\pm}\right\Vert _{L^{1}}$ and assume $Z_{\pm}>0$ (the
case when one or both are zero is easier). Call $p^{\pm}=Z_{\pm}^{-1}f^{\pm}$%
, elements of $L_{\text{dens}}^{1}\cap L_{2}^{1}$. By assumption, 
\begin{equation*}
\left\Vert \cl_{h}^{n}p^{\pm}\right\Vert _{L_{2}^{1}}\leq
C\lambda^{n}\left\Vert p^{\pm}\right\Vert _{L_{2}^{1}}+D 
\end{equation*}
hence, by linearity of $L_{h}^{n}$ and homogeneity of the norms, 
\begin{equation*}
\left\Vert \cl_{h}^{n}f^{\pm}\right\Vert _{L_{2}^{1}}\leq
C\lambda^{n}\left\Vert f^{\pm}\right\Vert _{L_{2}^{1}}+Z_{\pm}D. 
\end{equation*}
Again by linearity of $\cl_{h}^{n}$, 
\begin{align*}
\left\Vert \cl_{h}^{n}f\right\Vert _{L_{2}^{1}} & =\left\Vert \cl_{h}^{n}\left(
f^{+}-f^{-}\right) \right\Vert _{L_{2}^{1}}\leq\left\Vert
\cl_{h}^{n}f^{+}\right\Vert _{L_{2}^{1}}+\left\Vert \cl_{h}^{n}f^{-}\right\Vert
_{L_{2}^{1}} \\
& \leq C\lambda^{n}\left( \left\Vert f^{+}\right\Vert
_{L_{2}^{1}}+\left\Vert f^{-}\right\Vert _{L_{2}^{1}}\right) +\left(
\left\Vert f^{+}\right\Vert _{L^{1}}+\left\Vert f^{-}\right\Vert
_{L^{1}}\right) D \\
& \leq2C\lambda^{n}\left\Vert f\right\Vert _{L_{2}^{1}}+2D\left\Vert
f\right\Vert _{L^{1}}
\end{align*}
where in the last step we have used the fact that $\left\vert f^{\pm}\left(
x\right) \right\vert \leq\left\vert f\left( x\right) \right\vert $, and the
definition of the norms in $L_{2}^{1}$ and $L^{1}$. Therefore (\ref{ineq1
lemma 18}) holds with $A=2C$ and $B=2D$.

\textbf{Step 2}. Let us prove (\ref{ineq2 lemma 18}). Given $p\in L_{\text{%
dens}}^{1}\cap L_{2}^{1}$, on a suitable probability space choose a random
initial condition $X_{0}$ with density $p$, independent of the Brownian
motion. Call $X_{t}$ the solution of the Cauchy problem with initial
condition $X_{0}$. Recall that the Kolmogorov operator $P_{t}$ is defined by
means of the solutions $X_{t}\left( x\right) $ of the same Cauchy problem
but with deterministic initial condition $x$, $\left( P_{t}\phi \right)
\left( x\right) =\mathbb{E}\left[ \phi \left( X_{t}\left( x\right) \right) %
\right] $. A simple disintegration argument proves that 
\begin{equation}
\int_{\mathbb{R}^{d}}\left( P_{t}\phi \right) \left( x\right) p\left(
x\right) dx=\mathbb{E}\left[ \phi \left( X_{t}\right) \right] 
\label{Kolmogorov and random ic}
\end{equation}%
where $X_{t}$, as defined above, is the solution with initial condition $%
X_{0}$. Since in our case \ $X_{t}^{x}$ \ has density \ $S_t(x,y)$ absolutely continuous w.r.t to Lebesgue, we have 
\[
\mathbb{E}[\phi(X_t)]  =  \mathbb{E} \left[ \mathbb{E}[\phi(X_t)|X_0]   \right],
\]
since it holds 
\[
 \mathbb{E}[\phi(X_t)|X_0] = P_t  \phi  ( X_0 ).
\]
Since the law of $X_0$ is $p(x)dx$,  we rewrite the conditional expectation by the law
\[
   E[    (P_t  \phi)  ( X_0 )   ]  =  \int  (P_t  \phi)  ( x )  p(x)dx.
\]

We shall use now this relation together with the duality relation
between the Kolmogorov operator and the transfer operator. We have (recall
that, by Lemma \ref{ror}, $\cl_{h}^{n}=\cl_{hn}$) 
\begin{eqnarray*}
\left\Vert \cl_{h}^{n}p\right\Vert _{L_{2}^{1}} &=&\int_{\mathbb{R}^{d}}\left(
1+\left\vert x\right\vert ^{2}\right) \left\vert \left( \cl_{hn}p\right)
\left( x\right) \right\vert dx \\
&=&\int_{\mathbb{R}^{d}}\left( 1+\left\vert x\right\vert ^{2}\right) \left(
\cl_{hn}p\right) \left( x\right) dx
\end{eqnarray*}%
because $p$ is a probability density and $\cl_{hn}p$ is thus non-negative, 
\begin{equation*}
=\int_{\mathbb{R}^{d}}\left( P_{hn}\phi \right) \left( x\right) p\left(
x\right) dx
\end{equation*}%
by the duality relation mentioned above, where $\phi \left( x\right)
=1+\left\vert x\right\vert ^{2}$, 
\begin{equation*}
=\mathbb{E}\left[ \phi \left( X_{hn}\right) \right] 
\end{equation*}%
by (\ref{Kolmogorov and random ic}). We have thus proved that {\ } 
\begin{equation*}
\left\Vert \cl_{h}^{n}p\right\Vert _{L_{2}^{1}}=\mathbb{E}\left[ 1+\left\vert
X_{nh}\right\vert ^{2}\right] .
\end{equation*}%
Hence inequality (\ref{ineq2 lemma 18}) reduces to prove%
\begin{equation}
\mathbb{E}\left[ 1+\left\vert X_{nh}\right\vert ^{2}\right] \leq C\lambda
^{n}\mathbb{E}\left[ 1+\left\vert X_{0}\right\vert ^{2}\right] +D.
\label{ineq3 lemma 18}
\end{equation}

\textbf{Step 3}. In this step we prove the inequality
\begin{equation}
\mathbb{E}\left[ \left\vert X_{t}\right\vert ^{2}\right] \leq e^{-2tR_{2}}%
\mathbb{E}\left[ \left\vert X_{0}\right\vert ^{2}\right] +\frac{2R_{1}+d}{%
2R_{2}}  \label{ineq4 lemma 18}
\end{equation}%
where $R_{1},R_{2}$ are the constants in the assumption on $b$ 
and $d$ is the space dimension. It is straightforward to see that (\ref%
{ineq4 lemma 18}) implies (\ref{ineq3 lemma 18}), completing the proof of
the lemma.

It is well known that, when $\mathbb{E}\left[ \left\vert X_{0}\right\vert
^{2}\right] <\infty$, we have 
\begin{equation*}
\sup_{t\in\left[ 0,T\right] }\mathbb{E}\left[ \left\vert X_{t}\right\vert
^{2}\right] <\infty 
\end{equation*}
for every $T>0$ (also with the supremum inside the expectation). For completeness we give the proof of this statement in Step 4 below. Here we shall use this fact. By It\^{o} formula,%
\begin{equation*}
\left\vert X_{t}\right\vert ^{2}=\left\vert X_{0}\right\vert
^{2}+\int_{0}^{t}2\left\langle X_{s},b\left( X_{s}\right) \right\rangle
ds+\int_{0}^{t}2\left\langle X_{s},dW_{s}\right\rangle +\text{Tr}\left( I\right) t.
\end{equation*}%
Assume $\mathbb{E}\left[ \left\vert X_{0}\right\vert ^{2}\right] <\infty $.
Then $\mathbb{E}\int_{0}^{T}\left\vert X_{t}\right\vert ^{2}dt<\infty $ for
every $T>0$ and thus, by the properties of It\^{o} integrals, $\mathbb{E}%
\int_{0}^{t}2\left\langle X_{s},dW_{s}\right\rangle =0$. Then%
\begin{equation*}
\mathbb{E}\left[ \left\vert X_{t}\right\vert ^{2}\right] =\mathbb{E}\left[
\left\vert X_{0}\right\vert ^{2}\right] +\int_{0}^{t}2\mathbb{E}\left\langle
X_{s},b\left( X_{s}\right) \right\rangle ds+\text{Tr}\left( I\right) t.
\end{equation*}%
This identity and the fact that the function $s\mapsto \mathbb{E}%
\left\langle X_{s},b\left( X_{s}\right) \right\rangle $ is continuous, imply
that the function $t\mapsto \mathbb{E}\left[ \left\vert X_{t}\right\vert ^{2}%
\right] $ is of class $C^{1}$ and%
\begin{equation*}
\frac{d}{dt}\mathbb{E}\left[ \left\vert X_{t}\right\vert ^{2}\right] =2%
\mathbb{E}\left\langle X_{t},b\left( X_{t}\right) \right\rangle +d.
\end{equation*}%
From the assumptions on $b$, 
\begin{equation*}
\frac{d}{dt}\mathbb{E}\left[ \left\vert X_{t}\right\vert ^{2}\right] \leq
-2R_{2}\mathbb{E}\left[ \left\vert X_{t}\right\vert ^{2}\right] +2R_{1}+d.
\end{equation*}%
This implies%
\begin{align*}
\mathbb{E}\left[ \left\vert X_{t}\right\vert ^{2}\right] & \leq e^{-2R_{2}t}%
\mathbb{E}\left[ \left\vert X_{0}\right\vert ^{2}\right] +%
\int_{0}^{t}e^{-2R_{2}\left( t-s\right) }\left( 2R_{1}+d\right) ds \\
& \leq e^{-2R_{2}t}\mathbb{E}\left[ \left\vert X_{0}\right\vert ^{2}\right] +%
\frac{2R_{1}+d}{2R_{2}}.
\end{align*}

\textbf{Step 4}. Given $R>0$, let $\tau_{R}$ be the first time $\left\vert
X_{t}\right\vert $ exceeds $R$, infinity if this never happens. We have%
\begin{align*}
\left\vert X_{t\wedge\tau_{R}}\right\vert ^{2} & =\left\vert
X_{0}\right\vert ^{2}+\int_{0}^{t\wedge\tau_{R}}2\left\langle X_{s},b\left(
X_{s}\right) \right\rangle ds+\int_{0}^{t\wedge\tau_{R}}2\left\langle
X_{s},dW_{s}\right\rangle +\text{Tr}\left( I\right) t\wedge\tau_{R} \\
& =\left\vert X_{0}\right\vert ^{2}+\int_{0}^{t}1_{\left\{ s\leq\tau
_{R}\right\} }2\left\langle X_{s},b\left( X_{s}\right) \right\rangle
ds+\int_{0}^{t}1_{\left\{ s\leq\tau_{R}\right\} }2\left\langle
X_{s},dW_{s}\right\rangle +\text{Tr}\left( I\right) t\wedge\tau_{R}.
\end{align*}
Since $\mathbb{E}\int_{0}^{t}1_{\left\{ s\leq\tau_{R}\right\} }2\left\langle
X_{s},dW_{s}\right\rangle =0$ we get 
\begin{align*}
\mathbb{E}\left[ \left\vert X_{t\wedge\tau_{R}}\right\vert ^{2}\right] & \leq%
\mathbb{E}\left[ \left\vert X_{0}\right\vert ^{2}\right] +\mathbb{E}%
\int_{0}^{t}1_{\left\{ s\leq\tau_{R}\right\} }2\left\langle X_{s},b\left(
X_{s}\right) \right\rangle ds+\text{Tr}\left( I\right) t \\
& \leq\mathbb{E}\left[ \left\vert X_{0}\right\vert ^{2}\right] +\mathbb{E}%
\int_{0}^{t}1_{\left\{ s\leq\tau_{R}\right\} }\left( 2L_{2}\left\vert
X_{s}\right\vert ^{2}+2C_{1}\right) ds+\text{Tr}\left( I\right) t \\
& \leq\mathbb{E}\left[ \left\vert X_{0}\right\vert ^{2}\right] +\mathbb{E}%
\int_{0}^{t}\left( 2R_{2}\left\vert X_{s\wedge\tau_{R}}\right\vert
^{2}+2C_{1}\right) ds+\text{Tr}\left( I\right) t.
\end{align*}
By Gronwall lemma this implies, given any $T>0$,%
\begin{equation*}
\mathbb{E}\left[ \left\vert X_{t\wedge\tau_{R}}\right\vert ^{2}\right]
\leq\left( \mathbb{E}\left[ \left\vert X_{0}\right\vert ^{2}\right]
+2TC_{1}+Td\right) e^{2R_{2}T}=:C 
\end{equation*}
for every $t\in\left[ 0,T\right] $. By Fatou lemma, 
\begin{equation*}
\mathbb{E}\left[ \lim_{R\rightarrow\infty}\left\vert
X_{t\wedge\tau_{R}}\right\vert ^{2}\right] \leq C. 
\end{equation*}
Now, a.s., $\lim_{R\rightarrow\infty}\tau_{R}=+\infty$, because the solution 
$X_{t}$ exists globally. Hence $\lim_{R\rightarrow\infty}\left\vert
X_{t\wedge\tau_{R}}\right\vert ^{2}=\left\vert X_{t}\right\vert ^{2}$ and
the proof is complete.
\end{proof}

In the following Lemma we see how the presence of the noise, and then the possibility to see the transfer operator as a kernel operator, also provides a form of regularization.

\begin{lemma} \label{lem:Rdregularization}
For every $t>0$, $\cl_{t}$ is bounded linear from $L^{1}\left(
\bR^{d}\right)  $ to $C^{1}\left(  \bR^{d}\right)  $; in
particular there exists $C_{t}>0$ such that
\begin{equation} \label{C1}
\left\Vert \cl_{t}f\right\Vert _{C^{1}\left(  \bR^{d}\right)
}\leq C_{t}\left\Vert f\right\Vert _{L^{1}\left(  \bR^{d}\right)  }.
\end{equation}
Moreover if $f \in L^1_{2}$ then
\begin{equation} \label{eq:RdstrongLY}
\left\Vert \cl_{t}f\right\Vert _{BV_{2} }\leq  C_{t,2} \left\Vert f\right\Vert _{L^{1}_{2}\left(  \bR^{d}\right)  }.
\end{equation}
\end{lemma}
\begin{proof} The estimate
\eqref{C1} follows by the definition of the transfer operator \eqref{eq:stoctransfer}
and the estimate on its derivatives by the results of Menotti-Pesce-Zhang i.e. item 2 of Theorem \ref{MPZ}.
The estimate 
\eqref{eq:RdstrongLY}
follows from
\eqref{def:BValpha} considering that $\|\mathcal{L}_t(f)\|_{L^1_2}$ is estimated by Lemma \ref{lem:pbound} 

{  and for $f \in C^1(\mathbb{R}^d)$ it holds 
\begin{equation} \label{eq:osctoC1}
\begin{split}
\|f\|_{osc(\mathbb{R}^d)} & = \sup_{\epsilon\le 1} \epsilon^{-1} 
\int_{\bR^d} osc(f, B_{\epsilon}(x))d\psi(x)   \\
& \leq  2 \| f \|_{C^1} \int_{\bR^d} d\psi(x)  \leq 2 \|f\|_{C^1}
\end{split}
\end{equation}
since $\psi $ is a probability measure. 
}
Finally by \eqref{C1} we can bound $\|f\|_{osc(\mathbb{R}^d)}$ by the $L^1$ norm and then by the $L^1_2$ norm of $f$.
\end{proof} 
By Lemma  \ref{lem:pbound} 
 and Lemma \ref{lem:Rdregularization} we get

\begin{lemma}\label{200} Given $h>0$, there exist constants $A,B>0$ and $\lambda\in\left(  0,1\right)
$ (also depending on $h$) such that
\begin{equation}
\left\Vert \cl_{h}^{n}f\right\Vert _{BV_{2}}\leq A\lambda^{n}\left\Vert
f\right\Vert _{BV_{2}}+B\left\Vert f\right\Vert _{L^{1}}
\end{equation}
for every $f\in BV_{2}$ and every $n\in\mathbb{N}$.
\end{lemma} 
\begin{proof}
By Lemma \ref{lem:Rdregularization}
\begin{equation}
\left\Vert \cl L_{h}^{n}f\right\Vert _{BV_{2}}\leq C_{h,2} ||\cl_{h}^{n-1}f||_{L_2^1}
\end{equation}
and by Lemma \ref{lem:pbound} 
\begin{equation}
  ||\cl_{h}^{n-1}f||_{L_2^1}\leq
   A\lambda^{n-1}\left\Vert
f\right\Vert _{L_{2}^1}+B\left\Vert f\right\Vert _{L^{1}}. 
\end{equation}

Putting together these two inequalities we get the statement\footnote{To simplify notations we keep the same $A$ and $B$ as in Lemma  \ref{lem:pbound}, although they differ by a constant.}.
\end{proof}

\begin{remark}
As it is well known (see e.g \cite{HH})  the Lasota-Yorke-Doeblin-Fortet  inequality proved in Lemma \ref{200} together with the compact embedding proved in Theorem \ref{thm:embeddingRD} implies that the operator $\cl_t$ acting on $BV_2$ is quasicompact (see Section \ref{app:repfo} for more details).  \
The use of \eqref{eq:RdstrongLY} togheter with Lemma \ref{lem:pbound}, and Theorem \ref{thm:embeddingRD} also implies that $\mathcal{L}_t$ is quasicompact when acting on $L^1_2$.

It is worth to remark that since the phase space is $\mathbb{R}^d$ (and thus not bounded)   $\mathcal{L}_t$ is not in general a compact operator when acting  on $L^1_2$ and hence we cannot rely on a simple spectral perturbation theory for such operators.

For simplicity let us illustrate this in the particular case where  $b(x)=-x$ and $h$ small enough.
 We will find a bounded sequence $f_n\in L^1_2$  such that $\mathcal{L}_t f_n$ has no converging subsequences.
If $h$ is small enough for each $x_0\in \mathbb{R}^d$ by Theorem \ref{MPZ} one can find  a radius $r_0<1$ and a point $y_0$   such that if a  density $f_0$ with $||f_0||_{BV_2}=1$ is supported on $B(x_0,r_0)$ then $\cl_h(f_0) $
is such that $\int_{B(y_0,1)} \cl_h(f_0) dx\geq 0.9 \int_{\mathbb{R}^d} f_0 dx$
(think about $r_0$ being very small and $f_0 $ concentrated around a certain point $x_0$, then $\cl_h(f_0)$ will be concentrated in a certain neighborhood of $y_0=\theta_h(x_0)$ and the distribution will be dominated by gaussians depending on $h$).
Still by Theorem \ref{MPZ} and the fact that $b(x)=-x$ we can get that if $h$ is small enough we can choose $y_0$ such that $|y_0|\geq 0.7 |x_0|$.
Thus, if $|x_0|\geq 10$

\begin{equation}\int_{B(y_0,1)} \cl_h(f_0)[ |x|^2+1] dx\geq \frac 12 ||f_0||_{L_2^1}=\frac 12.
\end{equation}
Now we can find $f_1$ with support far enough from the support of $f_0$ such that: $||f_1||_{L^1_2}=1$, the support of $f_1$ is near to a point $x_1$ with $|x_1|\geq 10$ as before, there is a point $y_1$, such that
$\int_{B(y_1,1)} \cl_h(f_1) dx\geq 0.9 \int_{\mathbb{R}^d} f_1 dx$ and $B(y_1,1)\cap B(y_0,1)=\emptyset$. 
This mean that $||\cl_h(f_1)-\cl_h(f_0)||_{L^2_1}\geq \frac 12.$ One can then find similarly $f_2$
with support far enough from $f_0,f_1$ such that  $||\cl_h(f_j)-\cl_h(f_i)||_{L^2_1}\geq \frac 12 $  for $i\neq j\in \{0,1,2\}$ and so on, finding a sequence $f_i$ on the unit ball such that $\mathcal{L}_tf_i$ has no converging subsequences.
\end{remark}

\subsection{Regularization for the perturbed operators}
\label{rpo}
In this section we prove a uniform Lasota Yorke inequality for the perturbed operators with strong and weak spaces $BV_2, \ L^1_2$.

\begin{proposition} \label{prop:verR3_Rd}
 
There is $0<\lambda<1$  and two positive constants $A', B'$ such that for any  $f\in BV_{2}(\bR^d),$ $t>0,$ $n\ge 1$ and  $m\ge 1$, we have 

$$||\cl_{t,n}^m f||_{BV_2}\le A' \lambda^{m}||f||_{BV_2}+B'||f||_{L^1({\mathbb{R}^d}).}$$
\end{proposition}

The proof of the Proposition is based on the following two preliminary results.
\begin{proposition}
\label{correz1}There are $\lambda <1,$ $A,B\geq 0$\footnote{Again we keep the same $A$ and $B$ as in Lemma  \ref{lem:pbound}, although they differ by a constant.} such that for each $f\in
L_{2 }^{1}$, $m,n\geq 1$%
\begin{equation}
||\cl_{t,n}^{m}f||_{L_{2}^{1}}\leq A\lambda
^{m}||f||_{L_{2}^{1}}+B||f||_{L^{1}}.
\end{equation}
\end{proposition}

We first consider the simplified case in which $f\geq 0.$

\begin{lemma}
\label{correz}There are $\lambda <1,$ and $A$ and $B$ positive constants, such that for each  $f\in
L_{2 }^{1}$ with $f\geq 0$, $m,n\geq 1$%
\begin{equation}
||\cl_{t,n }^{m}f||_{L_{2}^{1}}\leq A\lambda
^{m}||f||_{L_{2}^{1}}+B||f||_{L^{1}}.  \label{2332}
\end{equation}
\end{lemma}

\begin{proof} 
Consider $f\geq 0$.  We have that $\cl_{t,n}f=1_{B^c_n}\mathcal{L}_t(f)$ and 
$\mathcal{L}_t f = \cl_{t,n}f+1_{B_n} \mathcal{L}_t(f)$. 
We define the sequence $g_n$  
  in the following way: $g_{m}=\mathcal{L}_t^{m}f- \cl_{t,n}^{m}f$.

Since $f\geq 0$ and clearly also $g_n\geq 0$ we have that 
\begin{equation*}
\|\cl_{t,n }^{m}f\|_{L_{2}^{1}}=\int |\cl_{t,n}^{m}f|\rho_2 dx\leq
\int (\cl_{t,n
}^{m}f+g_{m})\rho_2 dx=
||\mathcal{L}_t^{m}f||_{L_{2}^{1}},
\end{equation*}
by Lemma \ref{lem:pbound} we have then $(\ref{2332})$.
\end{proof}

\begin{proof}[Proof of Proposition \protect\ref{correz1}.]

Let $f=f^{+}-f^{-}$ be decomposed into its positive and negative part (where $f^{+},f^{-}\geq 0$). It holds%
\begin{equation*}
||f||_{L_{2}^{1}}=\int |(f^{+}-f^{-})|\rho_2
dx=||f^{+}||_{L_{2}^{1}}+||f^{-}||_{L_{2}^{1}}
\end{equation*}%
and by Lemma \ref{correz} we get%
\begin{eqnarray*}
||\cl_{t,n}^{m}f||_{L_{2}^{1}} &\leq &2A\lambda
^{n}(||f^{+}||_{L_{2}^{1}}+||f^{-}||_{L_{2}^{1}})+2B(||f^{+}||_{L^{1}}+||f^{-}||_{L^{1}})
\\
&\leq &2A\lambda ^{n}||f||_{L_{2}^{1}}+2B||f||_{L^{1}}.
\end{eqnarray*}
\end{proof}

\begin{proof}[ Proof of Proposition \ref{prop:verR3_Rd}]

The $BV_2$ norm is the sum of the oscillation part and of the $L^1_2$ part. For the latter we use Lemma \ref{correz}. For the oscillation we  integrate  \cite[Proposition 3.2 (ii)]{saussol00} with respect to our norm. Namely 

\begin{equation}\label{28}
\|1_{B^c_{n}}\mathcal{L}_tf\|_{osc}\le \sup_{0<\eta\le 1}\frac{1}{\eta}\int Osc(\mathcal{L}_tf, B_n^c\cap B_{\eta}(x)) {\bf 1}_{B_n^c}(x)d\psi(x)+
\end{equation}
\begin{equation}\label{ooji}
\sup_{0<\eta\le 1}\frac{1}{\eta}\int 2\left[  {\sup}_{B_{\eta}(x)\cap B^c_n} |\mathcal{L}_t f|\right]{\bf 1}_{B_{\eta}(B_n)\cap B_{\eta}(B_n^c)}(x) d\psi(x),
\end{equation} where $B_{\eta}(x)$ is a ball centered at $x$ and with radius $\eta$ and given a set $S,$ $B_{\eta}(S)=\{x; \text{dist}(x, S)\le \eta\}.$ There are now two cases. We suppose first than $\eta<e^{-u_n};$ in this case 
 the only points $x$ contributing to the rightmost integral in the previous inequality, are those belonging to a $2\eta$-closed neighborhood of the boundary of the ball $B_{n};$  we call $S_n$ such an annulus. 
The Lebesgue  measure of $S_n$ will be bounded 
 by a constant $\tilde{C}$ (depending on $d$) times $\eta.$\footnote{The volume of $S_n$ is bounded by  the volume of the corresponding  annulus of size $2\eta$  around an hypersphere of radius $1.$ In this case  the lowest order of such a volume is   $\eta$ times a constant depending on $d.$} In the second case, $\eta\ge e^{-u_n}$ only the points  belonging to $B_{2\eta}(B_n)$ will contribute to the integral and the measure of these points  is the volume of the hyper sphere of radius $\eta$ which is $O(\eta^d).$ 
 The term $\sup_{B_{\eta}(x)\cap B^c_n}|\mathcal{L}_tf|$ will be bounded using Lemma \ref{imm} 
 and \ref{lem:Rdregularization}, therefore 
by calling $\mathcal{K}_n$ the closed ball of radius $2e^{-u_n}$ and by using Proposition \ref{imm}, 
we  continue to bound the quantity in (\ref{ooji}) as

\begin{equation}\label{ttt}
\sup_{0<\eta\le 1}\frac{\tilde{C}\eta}{\eta} ||\mathcal{L}_tf||_{L^{\infty} (\psi, S_n)}\psi(S_n)\le
2\hat{C}\tilde{C} ||f||_{L^1_2(\mathbb{R}^d)}||\psi'||_{L^{\infty}(\mathbb{R}^d)},
\end{equation}
where we used the bound of order $\eta$ given by the first case, which includes also  the second case of higher order $\eta^{d-1}$.

The constant $\hat{C}$ maximizes  the constant on the right hand side of (\ref{in}), depending on $\psi$
and on its strictly positive infimum over $S_n$, and the constant entering formula (\ref{eq:RdstrongLY}).

The right-hand side of (\ref{28}) is bounded by $\|\mathcal{L}_t f\|_{osc}$ and the latter is again bounded as in (\ref{eq:RdstrongLY}). Therefore we get:
$$
||1_{B^c_{n}}\mathcal{L}_t f||_{osc}\le  (\hat{C}+2\hat{C}\tilde{C}||\psi'||_{L^{\infty}(\mathbb{R}^d)})||f||_{L^1_2(\mathbb{R}^d)}.$$

If we now iterate this one we have, calling $c^{*}=\hat{C}+2\hat{C}\tilde{C}||\psi'||_{L^{\infty}(\mathbb{R}^d)},$ it holds 
$$
||\cl_{t,n}^mf||_{osc}\le ||\cl_{t,n}L_{t,n}^{m-1}f||_{osc}\le c^* ||\cl_{t,n}^{m-1}f||_{L^1_2({\mathbb{R}^d})}
$$
and using Lemma \ref{correz} we continue as
$$
||\cl_{t,n}^mf||_{osc}\le c^* A\lambda^{m-1}||f||_{L^1_2({\mathbb{R}^d})}+c^* B||f||_{L^1({\mathbb{R}^d})}
$$
If we now define $A'=\max(c^*/\lambda, A), B'=\max(c^*,B),$ we finally have
\begin{equation}\label{jh}
||\cl_{t,n}^m f||_{BV_2}\le A' \lambda^{m}||f||_{L^1_2(\mathbb{R}^d)}+B'||f||_{L^1({\mathbb{R}^d})},
\end{equation}
which implies the desired result.\\
\end{proof}

\section{Rare Events Via Transfer Operator} \label{app:repfo}

For completeness, we remind here the statement of a result due to  Keller and Liverani \cite{KL09} which is fundamental to our general construction.  
We consider a Banach spaces $(\cB, \|\cdot\|)$ and we  denote with $\cB^*$ its dual.
Then let $\cL_{\epsilon} : \cB \to \cB $ be a family of uniformly bounded   linear operators, where $\epsilon\in E,$ and $E$ is the interval $E=(0,\overline{\epsilon}]$ for some $\overline{\epsilon}>0.$\footnote{In general the set $E$ could  be taken as a closed set of parameters with $\epsilon=0$ as an accumulation point. Moreover, as Keller wrote in \cite{Keller12} "It
is assumed that $E$ is a closed subset of $\mathbb{R}$, but again the parameter $\epsilon$ enters the estimates only via
the derived quantities $\pi_{\epsilon}$ and $\Delta_{\epsilon}$  ({\em see our conditions R2 and R3 below}), so that also this result is valid for more
general sets of parameters."}

\subsubsection{Perturbative hypotheses} \label{sec:perturbativereminder}

\begin{itemize}
\item[R1] The operators $\cl_{\epsilon},$ $\epsilon \in E,$ must satisfy  the spectral decomposition  
\begin{equation}\label{specdec}
\lambda_{\epsilon}^{-1}\cL_{\epsilon}=\varphi _{\epsilon}\otimes \nu _{\epsilon}+Q_{\epsilon}
\end{equation}
where $\lambda_{\epsilon}\in \mathbb{C},$ $\varphi _{\epsilon}\in  \cB ,\nu_{\epsilon}\in \cB^{*},$  $Q_{\epsilon}: \cB \rightarrow \cB$ is a linear operator verifying 
\begin{equation}\label{qqcc}
    \sum_{n=0}^{\infty}\sup_{\epsilon\in E}\|Q_{\epsilon}^n\|<\infty.
\end{equation}

Moreover, $\cl_{\epsilon}\varphi_{\epsilon}=\lambda_{\epsilon} \varphi_{\epsilon},$ $\nu_{\epsilon}\cl_{\epsilon}=\lambda_{\epsilon}\nu_{\epsilon},$ $\nu_{\epsilon}(\varphi_{\epsilon})=1,$ $\nu_{\epsilon}Q_{\epsilon}=0,$ $Q_{\epsilon}(\varphi_{\epsilon})=0.$\label{eq:B5}

We also require that $\nu_{0}(\phi_{\epsilon})=1$ and
\begin{equation}\label{coon}
\sup_{\epsilon \in E}\|\phi_{\epsilon}\|<\infty.
\end{equation}
\end{itemize}
\begin{itemize}
\item[R2] When $\varepsilon $ is small, $\cL_{\varepsilon }$ is a small
perturbation of $\cL_{0}$, in the following sense: 

\begin{equation*}
\pi_{\epsilon}:=\sup_{f\in \mathcal{B}, \|f\|\le 1}|\nu_0((\cl_0-\cl_{\epsilon})(f))|\rightarrow 0, \ \epsilon\rightarrow 0.
\end{equation*}
\end{itemize}

\begin{itemize}
\item[R3]
We now set
$$
\Delta_{\epsilon}:=\nu_0((\cl_0-\cl_{\epsilon})(\phi_0)). 
$$ Then we require that 
\begin{equation*}
\pi_{\epsilon}\|(\cL_0-\cL_{\epsilon})\varphi_0\|\le \ \text{constant} \ |\Delta_{\epsilon}|.
\end{equation*}

where the constant is independent of $\epsilon.$
\end{itemize}

\begin{itemize}
    \item[R4] 
Let us consider the following quantities
\begin{equation*}
q_{k,\varepsilon }:=\frac{\nu _{0}((\cL_{0}-\cL_{\varepsilon })\cL_{\varepsilon
}^{k}(\cL_{0}-\cL_{\varepsilon })(\varphi_0))}{ \Delta_{\epsilon}}.
\end{equation*}
We will assume  that for each $k\geq 0$ the following limit exists
\begin{equation}\label{eee}
\lim_{\epsilon \rightarrow 0}q_{k,\varepsilon }=q_k,
\end{equation}
and we define the \emph{extremal index of the system} as \begin{equation}\label{extremal}    
\theta=1-\sum_{k=0}^{\infty}q_k.
\end{equation}
\end{itemize}

\begin{proposition}[Proposition\cite{KL09}] \label{thm:repfo}
If $R1-R4$ are satisfied then 
\begin{equation} \label{eq:repfoexp}
\lambda_{\epsilon} = 1 - \theta \Delta_{\epsilon} + o( \Delta_{\epsilon}). \end{equation}
\end{proposition}

\subsubsection{Sufficient conditions to check assumptions R1 and R2}

\label{ssaa}
We now give sufficient conditions to check R1 and we will show how to rewrite R2 in a form adapted to our current setting. We begin to introduce a second Banach space $(\cB_w, \|\cdot\|_w)$ which we qualify as {\em weak} when compared with the {\em strong} Banach space $\cB.$ We will first suppose that
\begin{itemize}
\item[R0'] For all $ f \in \cB $ , $ \| f \|_w \leq \| f \|$ i.e. the weak norm is bounded by the strong norm and 
the unit ball of $\cB$ is compact in $\cB_w$. Moreover  $\exists G\geq 0$ s.t.
\begin{equation*}
\forall \epsilon \in E\ \forall f\in \cB,\ \forall n\in \mathbb{N}%
:\| \cL_{\epsilon }^{n}f\|_{w}\leqslant G\,\|f\|_{w}.  \label{eq:B1}
\end{equation*}
\end{itemize}
We then require that R1 be satisfied for $\epsilon=0.$ To achieve the same result for $\epsilon>0$ we need two ingredients which will allow us to apply another perturbative result by Keller and Liverani, \cite{KL2}.
The first ingredient asks that:
\begin{itemize}
\item[R1'] The operators $\cL_{\epsilon}$ satisfy a uniform, with respect to $\epsilon\in E$,  Lasota-Yorke (or Doeblin-Fortet) inequality:  there exists $\alpha \in (0,1)  ,D > 0 $ such that
\begin{equation} \label{eq:B2}
\forall \epsilon \in E\ \forall f\in \cB\ \forall n\in \mathbb{N}:\Vert
\cl_{\epsilon }^{n}f\Vert \leqslant D\alpha ^{n}\Vert f\Vert +D\|f\|_{w}
\end{equation}
\end{itemize}
Finally we ask the closeness of the operators in the so-called triple norm $|||\cdot|||$, namely there exists an upper semi-continuous function $u_{\epsilon}: [0, \infty)\rightarrow [0, \infty),$ $u_{\epsilon}>0,$ such that
\begin{itemize}
\item [R2']
\begin{equation}\label{vbn}
|||\cl_0-\cl_{\epsilon}|||:=\sup_{f\in \cB, \|f\|\le 1}\|(\cl_0-\cl_{\epsilon})(f)\|_w\le u_{\epsilon}\rightarrow 0, \ \epsilon\rightarrow 0.
\end{equation}
    \end{itemize}
  By assuming the previous conditions,  it now follows from \cite{KL2}, see also \cite{Keller12}, that the {\em quasi-compactness condition} (\ref{specdec}) holds for any $\epsilon \in E, \epsilon>0.$ Moreover there will be $0<\rho<1$ such that for any $\epsilon\in E,$ the spectral radius of $Q_{\epsilon}$ is bounded by $\rho,$ which implies (\ref{qqcc}).  Finally (\ref{coon}) follows immediately from (\ref{eq:B2}).\\

  Condition R2 could be easily worked out when the $\nu_0$ will be a measure because in this case we could write
  \begin{equation}\label{rtey}
  \pi_{\epsilon}\le \sup_{f\in \mathcal{B}, \|f\|\le 1}\int |(\cl-\cl_{\epsilon})(f)|d\nu_0.
  \end{equation}
If the weak norm is strong enough to bound this integral, condition R2' implies condition R2. More precisely, if we have that there is some $C$ such that 
$$\int |(\cl-\cl_{\epsilon})(f)|d\nu_0\leq C\|(\cl_0-\cl_{\epsilon})(f)\|_w$$ then R2' implies  R2.
We will see that in our case this holds.

\subsection{Verifying the perturbative assumptions in our case}
\label{sec:estimatesper}
\label{sec:Rdperturbation}

In this section we verify the perturbative assumptions needed to apply  Proposition \ref{thm:repfo}.
In order to do this, we will get information on the spectral picture of $\mathcal{L}_t$ when applied to $BV_2$ and
 show that the  perturbed operators $\cl_{t,n}$ 
(see \eqref{def:twistedoperator}) can be seen as a small perturbations in some sense of the operator $\mathcal{L}_t$ (see \eqref{eq:stoctransfer}).

We denote again by $B_n $ the ball $B(x_0, \exp(-u_n))$ and we remind that 
\begin{equation*}
(\cL_{t,n}f)(x)=  1_{B_{n }^{c}}(x) (\mathcal{L}_t f ) (x).
\end{equation*}
We now check that the perturbative assumptions listed in Sections \ref{sec:perturbativereminder} and \ref{ssaa} apply to these operators choosing $L^1$ as the weak norm and $ BV_{2}$ as the strong one.

\noindent {\bf Assumption R0'} comes directly from the definitions of the norms given in Section \ref{sec:spacesRD}, from the compact embedding result  proved at Theorem \ref{thm:embeddingRD} and finally from Lemma \ref{7}

\noindent {\bf Assumption R1 }. We first prove it for $\epsilon=0,$ then we extend it to $\epsilon \in E, \epsilon>0$ once R1' will be  proved, see below. Assumption R1 for $\epsilon=0$ is verified by classical results on the spectral picture of regularizing linear opertators.
The compact immersion proved at Theorem \ref{thm:embeddingRD} and the Lasota Yorke inequality proved in Lemma \ref{200} for $\cl_t,$ allows us to apply  the Ionescu-Tulcea-Marinescu theorem  (see for instance \cite{HH}), and therefore the  operator $\mathcal{L}_t$ has the following spectral decomposition $\mathcal{L}_t=\sum_{i}\upsilon_i\Pi_i+Q,$
where all $\upsilon_i$ are eigenvalues of $\cl_t$ of modulus $1$, $\Pi_i$ are finite-rank projectors onto the associated
eigenspaces, $Q$ is a bounded operator with a spectral radius strictly less than $1$. They satisfy
$\Pi_i \Pi_j = \delta_{\delta_{ij}}\Pi_i, Q\Pi_i = \Pi_iQ = 0$.

By this theorem we also get that   $1$ is an eigenvalue and  therefore the transfer operator will admit finitely many  absolutely continuous stationary measures. Furthermore, the peripheral spectrum is completely cyclic. Then we need to show that $1$ is a simple eigenvalue of
$\mathcal{L}_t$ and that there is no other peripheral eigenvalue. This is achieved by the strict positivity of the Markov kernel $S_t$ (see \eqref{eq:stoctransfer}) provided  by Theorem \ref{MPZ}. \footnote{Here is the argument. Consider first the operator
  $
  \cl_t f(y)=\int S_t(x,y)f(x)dx,
  $
  where $S_t(\cdot, \cdot)$ is the Markov kernel defined on $\mathbb{R}^2$ and is bounded from below on any compact set $K\in \mathbb{R}^2.$ We want to show that there exists only one fixed point for $\cl_t.$ Suppose there are two, say $h_1, h_2.$ Define
  $
  \hat{h}=\min(h_1, h_2).
  $
  Notice that $h_1-\hat{h}$ is nonnegative on a set $\Omega$ of positive Lebesgue measure; moreover from a classical trick, we have easily that
  $
  \cl_t(h_1-\hat{h})=h_1-\hat{h}.
  $
  Take now a sequence of monotonically increasing compact sets $K_n$ such that the first verifies $K_1\subset \Omega,$ $\text{Leb}(K_1\cap \Omega)>0$ and  $\bigcup_n K_n=\mathbb{R}^2$ (use the regularity of the Lebesgue measure).
  Then take $y\in K_n$ and we get
  $
  h_1(y)-\hat{h}(y)\ge \int_{x\in K_n\cap \Omega}S_t(x,y) (h_1(x)-\hat{h}(x))dx.
  $
  Notice that the right hand side is strictly positive. So we have that
  $
  h_1(y)>h_2(y), \ y\in K_n\ \Rightarrow  \int_{K_n}h_1 dx> \int_{K_n}h_2 dx.
    $
    By passing to the limit we have $ \int h_1 dx > \int h_2 dx,$
    which is impossible. Call  $h^*$ the unique fixed point of $\cl_t.$ Consider now the peripheral spectrum of $\cl_t$ which  consists of a finite union of finite cyclic groups. For each of those eigenvalues call $g$ one of the corresponding eigenvector. 
There exists $k \ge 1$ such that $1$ is the unique peripheral eigenvalue of $\cl_t^k:$ $\cl_t^kg=g$.   But $\cl_t^ kh^*=h^*,$ and by repeating the proof above we see that $g=h^*.$ 
}Notice that the projector  $\Pi_1$ will be the linear functional $\nu_0$ in the assumption R1.

 \noindent {\bf Assumption R1'} verified in Proposition \ref{prop:verR3_Rd}.

\noindent {\bf Assumption R2'}  is verified in the following proposition

\begin{proposition}\label{prop:R5}
The operator  $\cL_{t,n}$ is a small
perturbation of $\cL_{t}$, in the following sense: there is a monotone sequence $\pi_{n} \to 0$such
that 
\begin{equation*}||\cL_{t,n}-\cL_{t}||_{BV_{2}\to L^1}\leq \pi_n.
\end{equation*}

\end{proposition}
\begin{proof}
We have
$||\cL_{t,n}-\cL_{t}||_{L^1}=\|1_{B_{n}}(\mathcal{L}_tf)\|_{L^{1}} \leq \text{Leb}(B_n)||\mathcal{L}_t f||_\infty.$  Notice that by Lemma \ref{lem:Rdregularization} the $C^1(\mathbb{R})$ norm of $\mathcal{L}_tf,$ and therefore its $L^\infty$ norm is bounded by $C_t\|f\|_{L^1(\mathbb{R}^d)}$. We could finally take $\pi_n=C_t\text{Leb}(B_n).$\\

\end{proof}
Since in our case $\nu_0$ is nothing but the integral with respect to the Lebesgue measure, and the weak norm is the $L^1$ norm, R2' implies R2 as remarked at the end of Section \ref{ssaa}.

\noindent {\bf Assumption R3} in our setting  can be stated in the following way:

\begin{proposition}\label{prop:R6}
  Let $f_0$  denote the invariant density of the invariant measure $\mu;$
  then there exists a constant $C'$ such that
  $$
  \pi_n\|(\mathcal{L}_t-\cL_{t,n})f_0\|_{BV_2}\le C' \mu(B_n).
$$
\end{proposition}

We saw above that $\pi_n$ could be taken as  $C_t\text{Leb}(B_n); $ 
moreover by Proposition \ref{prop:verR3_Rd} the quantity $\|(\mathcal{L}_t-\cL_{t,n})f_0\|_{BV_2}=\|1_{B_n}\mathcal{L}_t f_0\|_{BV_2}$ is bounded by a constant $\tilde{K}.$ Therefore $ \pi_n\|(\mathcal{L}_t-\cL_{t,n})f_0\|_{BV_2}\le C_t\tilde{K}\text{Leb}(B_n)\le C_t\tilde{K}\frac{\mu(B_n)}{\inf_{B_n}f_0},$ which gives the desired result since  the  density $f_0$ is strictly positive on any finite domain of $\mathbb{R}^d,$ as it is proved in the next lemma.

\begin{lemma} \label{lem:Lemma3_posdensity}
Let $f_{0}$ be the density of the unique invariant measure. Then%
\[
\inf_{x\in B}f_{0}\left(  x\right)  >0,
\]
where $B$ is any bounded subset of $\mathbb{R}^d.$\\

\end{lemma}
\begin{proof}
For $t>0,$ part $1$ of Theorem \ref{MPZ} gives, for $(x,y)\in \mathbb{R}^d:$
$$
D(x,y):=C_{0}^{-1}g_{\lambda _{0}^{-1}}(t,\theta _{t}(x)-y)\leq S_t(x,y).
$$
The invariant measure $\mu = f_0 dx $ satisfies
$f_{0}\left(  x\right)  =\int_{\mathbb{R}^{d}}S_{t}\left(  x,y\right)
f_{0}\left(  y\right)  dy.$  Then take a compact $\hat{K}$ such that $\int_{\hat{K}}f_0dy>0.5,$ which is possible by the regularity of the Lebesgue measure.
Therefore, for $x\in B$
$$
f_{0}\left(  x\right)  \geq \min_{x\in B, y\in \hat{K}}D(x,y)\int_{\hat{K}}f_{0}(y)  dy>0.5 \min_{x\in B, y\in \hat{K}}D(x,y)
$$
and the minimum on $D$ is strictly positive by the smoothness of $g_{\lambda_0}^{-1}.$

\end{proof}
\noindent {\bf Assumption R4} needs some more work.  The next proposition will show that all the quantities $q_k$ defined in the Assumption R7 are equal to $0.$
In the present setting they are defined as the limit for $n\rightarrow \infty$ of the following quantities: 
$$
q_{k,n}=\frac{\int (\mathcal{L}_t-\cL_{t,n})\cL^k_{t,n}(\mathcal{L}_t-\cL_{t,n})(f_0)dm }{\mu(B_n)}
$$
where $m$ is the Lebesgue measure on $\mathbb{R}^d$.
We recall that in the previous definition $t$ is fixed by the initial choice of the time discretization  $t=h$, hence $q_{k,n}$ does not depend on $t$.
Before showing that all the $q_{k}$ are zero, we could easily prove that the extremal index $\theta$ is always  non-negative.
\begin{proposition}
 If the limits (\ref{eee}) exist, then 
 $$
 \sum_{k=0}^{\infty}q_k\le 1.
 $$
\end{proposition}
\begin{proof}
Since the $q_{k,n}$ are non-negative it would be enough to prove that $\sum_{k=0}^{\infty}q_{k,n}\le 1$ almost surely.
We show how to compute explicitly $q_{2,n},$ the generalization to all $k>2$ being straightforward.
  We have
  $$
  q_{2,n}=\frac{\int (\mathcal{L}_t-\cL_{t,n}) \mathcal{L}^2_{t,n}(\mathcal{L}_t-\cL_{t,n})f_0(x)dx}{\mu(B_n)}.
  $$
  By using three times the duality between the operators $\mathcal{L}_t$ and $P_t$ and the fact that $f_0$ is the fixed point of $\mathcal{L}_t$ we have for the numerator :
  
$$
\int P_t[P_t[P_t(1_{B_n}) \cdot 1_{B_n^c}]\cdot 1_{B_n^c}](x)1_{B_n}(x)f_0(x) dx.
  $$

  At this point it easy to see that  Lemma 12 can be rewritten in terms of the operators $P_t$ and  it reads for any $\phi \in L^\infty\left( \bR^d \right)$:
 \[
P_{t}\left(  P_{s}\left(
\phi\right)  \right)  \left(  x\right)  =\mathbb{E}\left[  
\phi\left(  X_{t+s}^{x}\right)  \right].
\] 
 In the same way one generalizes Corollary \ref{cor:EtoTwistedP}; 
  therefore we have in our case, by repeatedly using the Markov property:
  \begin{eqnarray*}
 P_t[P_t[P_t(1_{B_n}) \cdot 1_{B_n^c}]\cdot 1_{B_n^c}](x)=\\
\mathbb{E}[1_{B_n^c}(X^x_t)\mathbb{E}[1_{B_n^c}(X^y_{t}))\mathbb{E}[ 1_{B_n}(X^z_{t})]_{z=X^y_{t}}]_{y=X^y_{t}}]=\\
\mathbb{E}[1_{B_n^c}(X^x_{t})1_{B_n^c}(X^x_{2t})1_{B_n}(X^x_{3t})]
\end{eqnarray*} 
 By inserting in the integral above  we get
 $$
q_{2,n}=\frac{\bigintsss \mathbb{E}[1_{B_n^c}(X^x_{t})1_{B_n^c}(X^x_{2t})1_{B_n}(X^x_{3t})] 1_{B_n}(x)d\mu(x)}{\mu(B_n)}.
 $$
 Since the last characteristic function $1_{B_n}$ in the integrand does not depend on the noise, we could rewrite $q_{2,n}$ as
$$
q_{2,n}=\mathbb{E}\left[\frac{\bigintsss 1_{B_n^c}(X^x_{t})1_{B_n^c}(X^x_{2t})1_{B_n}(X^x_{3t}) 1_{B_n}(x)d\mu(x)}{\mu(B_n)}\right].
$$
 For almost all $\omega\in \Omega$ the term in the square bracket gives the relative $\mu$-measure of the points which are in $B_n,$ stay twice outside it and come back for the first time at the third step. This generalizes to the other $q_{k,n}$ which will give the relative measure of the measurable  events $$
 E_{k,n}:=\{\text{points which start in}\ B_n, \ \text{stay outside} \ k \ \text{times and return at the}\  k+1 \ \text{time}\}.$$ 
 
 Of course all the events $E_{k,n}$ are mutually  disjoint and therefore 
 $$
 \sum_{k=0}^{\infty}q_{k,n}\le 1.
 $$
 \end{proof}

\begin{proposition}\label{prop:R7}

For each $k\geq 0$ we have
\begin{equation*}
\lim_{n \rightarrow \infty}q_{k,n }=0.
\end{equation*}    
\end{proposition}
\begin{proof}
Let us introduce the function $f_{k,n}:=\frac{\cL^k_{t,n}(\mathcal{L}_t-\cL_{t,n})(f_0)}{\mu(B_n)};$ we have
$$
\lim_{n\rightarrow \infty}q_{k,n } = \lim_{n
\rightarrow \infty}\int (\mathcal{L}_t-\cL_{t,n })f_{k,n } dm
= \lim_{n \rightarrow \infty}\int 1_{B_{n }}\mathcal{L}_t f_{k,n }dm\le \lim_{n \rightarrow \infty}m(B_n)\|\mathcal{L}_t f_{k,n}\|_{\infty}.
$$

As in the proof of Proposition \ref{prop:R5},  by Lemma \ref{lem:Rdregularization},  the $C^1(\mathbb{R}^d)$ norm of $\mathcal{L}_tf_{k,n},$ and therefore its infinity norm, is bounded by $C_t\|f_{k,n}\|_{L^1(\mathbb{R}^d)}$. On the other hand  $C_t\|f_{k,n}\|_{L^1(\mathbb{R}^d)}\le  \frac{  C_t \text{Leb}(B_n)}{\mu(B_n)}$ by proposition \ref{prop:R5}. 
The ratio $\frac{\text{Leb}(B_n)}{\mu(B_n)}$  is  bounded by a constant, as shown in the proof of Proposition \ref{prop:R6}.
Then we get $$\lim_{n\rightarrow \infty}q_{k,n }\leq \text{Leb}(B_n)\frac{  C_t \text{Leb}(B_n)}{\mu(B_n)}\to 0.$$

\end{proof}

\subsection{Proof of Theorem \ref{thm:gen2}}\label{Pgen2}\label{app:repfo2}
In this section we can collect all the previous estimates and finally prove the main result of the paper.

\begin{proof}[Proof of Theorem \ref{thm:gen2}]

We recall some notations. 
 We consider   a given point in the phase  space  $x_0$  and a  sequence $u_n$   going to $0;$ then we denote by $B_{n}$ the ball $B\left(  x_{0},\exp\left(  -u_{n}\right)  \right)$. To compact notation we will write $g$ for $g_{x_0}:=-\log(d(x,x_0))$. We will denote with $\mathcal{L}_t$ the unperturbed transfer operator at time $t$ (see \eqref{eq:stoctransfer}) and with $\cl_{t,n}$  the perturbed transfer operator (see \eqref{def:twistedoperator}).    Let us also denote as $h>0$  the time  discretization step introduced in (\ref{TS}).

 Now, we rewrite \eqref{problem}, using  the notation above and remembering that $t_k=kh,$ as
{ 
\begin{equation} \label{eq:fromPtoE} \begin{aligned}
 & \mathbb{P} \otimes \mu \left(  \left\| X_{t_{k}}^x- x_0 \right\|
\leq\exp\left(  -u_{n}\right)  \text{ for every }k=0, \ldots ,n-1  \right)  \\
&  =\mathbb{P}\otimes \mu \left(  X_{t_{k}}^x\in B_{n}^{c}\text{ for every }%
k=0,\ldots,n -1 \right)  \\
 &= \int_{\Omega \times \mathbb{R}^d}  
1_{B_{n}^{c}} \left(X_{t_{0}}^x (\omega)\right)   \cdots
  1_{B_{n}^{c}} \left(  X_{t_{n-1}}^x(\omega)\right) d\omega d\mu(x)   \\
& = \int_{\mathbb{R}^d} 
\bE\left[ 1_{B_{n}^{c}} \left(X_{t_{0}}^x \right)   \cdots
  1_{B_{n}^{c}} \left(  X_{t_{n-1}}^x \right) \right] d\mu(x)  .
\end{aligned}
\end{equation}
}
Thus we can reformulate (\ref{problem}) by identifying sequences $\{ u_n \}_{n \in \bN}$ such that 
\begin{equation}
\lim_{n\rightarrow\infty}\int_{\mathbb{R}^d} 
\bE\left[ 1_{B_{n}^{c}} \left(X_{t_{0}}^x \right)   \cdots
  1_{B_{n}^{c}} \left(  X_{t_{n-1}}^x \right) \right] d\mu(x)   
\in\left(  0,1\right)  .\label{problem 1}%
\end{equation}

By using  Corollary \ref{cor:EtoTwistedP} together with (\ref{mod}), after recalling that the invariant measure $\mu$ is absolutely continuous w.r.t to Lebesgue with   density $f_{0}$, we are finally able to write the distribution of the maxima in an operator-like way as
\begin{equation} \label{pe}
\begin{split}
& \mathbb{P}\otimes \mu \left( \left( \max_{k=0,...,n-1}g_{x_0} \left(  X_{t_k}^x \right)\right) \leq u_{n} \right) \\
& =\int_{D} \bE\left[ 1_{B_{n}^{c}} (X_{t_{0}}^x )   \cdots
  1_{B_{n}^{c}}   (X_{t_{n-1}}^x )1(X_{t_{n-1}}^x ) \right] f_0(x) dx \\
  & = \int \left(P_{t_{0},n}\circ P_{t_{1}-t_{0},n}\circ\cdot\cdot\cdot\circ P_{t_{n}-t_{n-1},n }\right)(1)(x)f_0(x)dx=
\int \cl_{h,n}^{n}f_0dx
\end{split}
\end{equation}  
where $\cl_{h,n}^{n}$ denotes the $n$-th power of the operator $\cl_{h,n}$. \\

Now we apply Proposition \ref{thm:repfo} to the transfer operator $\cL_h$ and to its perturbations $\cL_{h,n}$.
We consider as a strong space the space $\cB(\bR^d) = BV_{2}$ and $L^1(\bR^d)$ as a weak space.
The assumptions of section \ref{sec:perturbativereminder} are verified by $\cL_h$ and by  $\cL_{h,n}$ thanks to the sufficient conditions quoted in   section \ref{sec:Rdperturbation}; therefore
we have the following spectral decomposition for the perturbed operator $\cL_{h,n}:$

\begin{equation} \label{eq:decomposition}
 \lambda_{h,n}^{-1}  \cL_{h,n} = f_{h,n} \otimes \mu_{h,n} + Q_{h,n}
\end{equation}
where  $ f_{h,n}\in BV_2 ,  \  \mu_{h,n} \in  (BV_2 )'$ and $  Q_{h,n} : BV_2  \to BV_2$ is a bounded operator with spectral radius uniformly bounded in $n$ by some $\rho<1.$ Moreover $\cl_{h,n}f_{h,n}=\lambda_{h,n}f_{h,n}$ and $\mu_{h,n}\cl_{h,n}=\lambda_{h,n}\mu_{h,n}.$\\

By denoting with $\langle \mu_{h,n},g\rangle$ the action of the linear functional $\mu_{h,n}$ over $g\in BV_2,$ we have

\begin{equation}\label{eq:perturbation}
\cl_{h,n} g=\lambda_{h,n} f_{h,n} \langle \mu_{h,n},g\rangle + \lambda_{h,n} Q_{h,n} (g);
\end{equation}
moreover we use the normalization $\int f_{h,n} dx=1$ and $\langle \mu_{h,n},f_{h,n}\rangle=1.$

Thus it is sufficient to control $\int (\cl_{h,n}^{n} f_0)(x) dx $  by plugging \eqref{eq:perturbation} in it.
Since we have a direct sum decomposition of our operator, we can iterate and get
\begin{equation} \label{eq:naction}
\int (\cl_{h,n}^{n} f_0)(x) \ dx =\lambda^{n}_{h,n} \langle \mu_{h,n},f_0\rangle    + \lambda^{n}_{h,n}\int (Q^n_{h,n} f_0)(x)\ dx.
\end{equation}

Now,   $1$  is the largest unique eigenvalue of  the unperturbed operator $\cl_h$   and by proposition \ref{thm:repfo} and  proposition \ref{prop:R7} we have $\theta=1,$ therefore:
\begin{equation} \label{eq:leadingperturb}
\lambda_n=1- \mu(B_n)+o(\mu(B_n)).
\end{equation}

Then by substituting \eqref{eq:leadingperturb} in \eqref{eq:naction} we have
\begin{equation} \label{eq:perturbation3}
\begin{split}
\int (\cl_{h,n}^{n} f_0)(x) \ dx & =e^{n\log(1- \mu(B_n)+o(\mu(B_n))}\left[\langle \mu_{h,n},f_0\rangle + \int Q^n_{h,n}f_0\ dx \right]\\
& =e^{-n \mu_0(B_n)+no(\mu_0(B_n))}\left[\langle \mu_n,f_0\rangle  + \int Q^n_{h,n} f_0\ dx\right].
\end{split}
\end{equation}

Let us now recall that by \cite[Lemma 6.1]{KL09} we have  $ \langle \mu_n,f_0\rangle \rightarrow 1.$ 
Moreover by \eqref{qqcc} we get a uniform   exponential  convergence to zero of $$ \int Q^n_{h,n} f_0\ dx \leq \|Q^n_{h,n} f_0\|_{L^1}\le \|Q^n_{h,n} f_0\|_{BV_2} \le \text{Const} \ \rho^n.$$  
 Now, as assumed in the statement of Proposition \ref{thm:gen2}, we choose  the sequence $\{ u_n \}_{n \in \bN}$ and a $\tau \in \bR, \tau > 0$ such that  
\begin{equation} \label{eq:hyplimit}
n\, \mu(B_n)\rightarrow \tau.
\end{equation} 
  Thus
\begin{equation} \label{eq:stop}
\int \cl_{h,n}^n f_0 \ dx\rightarrow e^{-\tau}
\end{equation}
proving \eqref{th2}.\\

 \end{proof}

\section{Poisson statistics}\label{sec:Poisson}
\begin{proof}[Proos of Theorem \ref{thm:Poisson}]
For this proof, similarly to the proof of Theorem \ref{thm:gen2} we will apply proposition \ref{thm:repfo}.
We start by computing the characteristic function of the random variable $S_{n}=\sum_{i=0}^{n-1}1_{B_n}(X^x_{ih}):$
\begin{equation}\label{CCFF}
\Phi_n(s)=\int e^{isS_{n}}d\mathbb{P}d\mu=\int \mathbb{E}(e^{isS_{n}})f_0 dx,
\end{equation}
where $f_0$ is the density of $\mu.$ We then introduce the perturbed operators for $f\in BV_2(\mathbb{R}^d):$\footnote{We now require that our functions have valued in $\mathbb{C};$  in this case  the oscillation is defined as $\text{osc}(f, S)=\text{esssup}|f(x)-f(y)|, x,y\in S.$ With this definition the  function spaces and the norms considered extend straightforwardly to the complex case. We will keep denoting these spaces as $L^1_{\alpha}(\mathbb{R}^d)$ and $BV_{\alpha}(\mathbb{R}^d)$ also in the complex valued case.}

\begin{equation}\label{NPO}
\mathcal{L}_{t,n} f(x)=e^{is1_{B_n}(x)}\mathcal{L}_tf(x), 
\end{equation}
\begin{equation}\label{NPO2}
\mathcal{P}_{t,n}f(x)=\mathbb{E}(e^{is1_{B_n}(X^x_t)}f(X^x_t))
\end{equation}
By using  (\ref{eq:stoctransfer})          we get,   for $f\in BV_2(\mathbb{R}^d), g \in L^{\infty}:$
\begin{equation}\label{hhhh}
\int \mathbb{E}(e^{is1_{B_n}(X^x_t)}g(X^x_t))f(x)dx=\int \mathcal{P}_{t,n} g(x) f(x)dx=\int \mathcal{L}_{t,n} f(x)g(x) dx.
\end{equation}

We then observe that Lemma \ref{ror} holds for the new operator $\mathcal{P}_{t,n}$  just by replacing  the characteristic function $1_{B_n}$ with  $e^{is1_{B_n}}.$ 
\begin{lemma}
For every $t,s\geq0$,  $\phi \in L^\infty\left( \bR^d \right)$   
\[
\mathcal{P}_{t,n}\left(  \mathcal{P}_{s,n}\left(
\phi\right)  \right)  \left(  x\right)  =\mathbb{E}\left[  e^{is1_{B_n}\left(  X_{t}^{x}\right)}%
 e^{is1_{B_n}\left(  X_{t+s}^{x}\right)}
\phi\left(  X_{t+s}^{x}\right)  \right].
\]
\end{lemma}

Using this and the equalities (\ref{hhhh}), we have
\begin{equation}\label{fedeu}
\Phi_n(s)=\int e^{isS_{n}}d\mathbb{P}d\mu=\int \mathcal{L}_{t,n}^nf_0(x)dx.
    \end{equation}

We now apply Theorem \ref{thm:repfo} to the operators $\mathcal{L}_{t,n},$ using $BV_2$ and $L^1$ as a strong and weak space. In order to do this, we check the assumptions (R1) to (R4), in particular (R1'), (R2'), (R3), (R4).
We will omit some detail, as the proofs are similar to what is done in Section \ref{sec:Rdperturbation}.
\begin{itemize}
\item (R1')  Lasota-Yorke inequality. We follow Propositions  \ref{prop:verR3_Rd} and  Proposition \ref{correz1}.  Let first compute the norm $\|\mathcal{L}^m_{t,n}f\|_{L^1_2},\ m\ge 1.$
 We have
 $$
 \|\mathcal{L}^m_{t,n}f\|_{L^1_2}= \|\mathcal{L}_{t,n}\mathcal{L}^{m-1}_{t,n}f\|_{L^1_2}=\|\mathcal{L}_t\mathcal{L}^{m-1}_{t,n}f\|_{L^1_2}
 \le 
 $$
 $$
\|\mathcal{L}_t|\mathcal{L}^{m-1}_{t,n}f|\|_{L^1_2|}\le 
\|\mathcal{L}_t^2|\mathcal{L}^{m-2}_{t,n}f|\|_{L^1_2}\le \dots \le  \|\mathcal{L}_t^m |f|\|_{L^1_2}.
$$
We finally apply lemma \ref{lem:pbound} to the rightmost quantity  to get the equivalent of Proposition \ref{correz1}.
   We now need the equivalent of Proposition \ref{prop:verR3_Rd}, in particular we have to estimate the oscillation seminorm of $e^{is1_{B_n}}\mathcal{L}_tf.$
   Using the formula\footnote{ If $u,v\in BV_2$  and $B$ a measurable set, then  $\text{osc}(uv, B)\le \text{osc}(u,B)\text{esssup}_B v+\text{osc}(v,B)\text{essinf}_B|u|.$}, we get
   $$
   \|e^{is1_{B_n}}\mathcal{L}_tf\|_{\text{osc}(\mathbb{R}^d)}\le \sup_{0<\eta\le1}\frac{1}{\eta}\int\text{osc}(\mathcal{L}_tf, B_{\eta}(x)) d\psi+
   $$
   $$\sup_{0<\eta\le1}\frac{1}{\eta}\int \text{osc}(e^{is1_{B_n}}, B_{\eta}(x)) \sup|\cl_tf|d\psi:=(I)+(II)
   $$
   The first  piece (I) on the right hand side is $\|\mathcal{L}_tf\|_{\text{osc}(\mathbb{R}^d)}$ and is bounded as in (\ref{eq:RdstrongLY}). For the second one, we now distinguish 
   two cases. We suppose first than $\eta<e^{-u_n};$
   the oscillation in the integral will contribute only when the balls $B_{\eta}(x)$ will cross at the same time the ball $B_n$ and its complement. In this case the oscillation will be constant and equal to $|e^{is}-1|$ on the $2\eta$-neighborhood of the ball $B_n.$ If we call $S_n$ such a neighborhood as   in the proof of Proposition \ref{prop:verR3_Rd}, the Lebesgue  measure of $S_n$ will be bounded 
 by a constant $\tilde{C}$ (depending on $d$) times $\eta.$ In the second case, $\eta\ge e^{-u_n}$ only the points  belonging to $B_{2\eta}(B_n)$ will contribute to the integral and the measure of these points  is the volume of the hyper sphere of radius $\eta$ which is $O(\eta^d).$ Then we get as in (\ref{ttt}), where we also introduced the constant $\hat{C}:$
 $$
 (II)\le \sup_{0<\eta\le 1}\frac{\tilde{C}\eta}{\eta} ||\mathcal{L}_tf||_{L^{\infty} (\psi, S_n)}\psi(S_n)\le
2\hat{C}\tilde{C} ||f||_{L^1_2(\mathbb{R}^d)}||\psi'||_{L^{\infty}(\mathbb{R}^d)}
 $$
 Therefore we get
 $$
 \|e^{is1_{B_n}}\mathcal{L}_tf\|_{\text{osc}(\mathbb{R}^d)}\le C_{t,2}  \|\mathcal{L}_tf\|_{\text{osc}(\mathbb{R}^d)}+\overline{C}||f||_{L^1_2(\mathbb{R}^d)}\le[ C_{t,2} +\overline{C}]\left\Vert f\right\Vert _{L^{1}_{2}\left(  \bR^{d}\right)  },
 $$
 where $\overline{C}=2\hat{C}\tilde{C} ||\psi'||_{L^{\infty}(\mathbb{R}^d)}.$ We then put $C^{@}:= C_{t,2} +\overline{C}$ and  we continue as
 $$
 \|\mathcal{L}^m_{t,n}f\|_{\text{osc}(\mathbb{R}^d)}= \|\mathcal{L}_{t,n}\mathcal{L}^{m-1}_{t,n}f\|_{\text{osc}(\mathbb{R}^d)}=\|e^{is1_{B_n}}\mathcal{L}_t(\mathcal{L}^{m-1}_{t,n}f)\|_{\text{osc}(\mathbb{R}^d)}\le
 $$
 $$
C^{@}\|\mathcal{L}^m_{t,n}f\|_{L^{1}_{2}\left(  \bR^{d}\right)}\le C^{@} \|\mathcal{L}_t^{m-1} |f|\|_{L^1_2}.
 $$
 We therefore bound the right hand side of the previous inequality with lemma \ref{lem:pbound} and this will finally give us the Lasota-Yorke inequality, since $\|\mathcal{L}^m_{t,n}f\|_{L^{1}\left(  \bR^{d}\right)}\le \|f\|_{L^{1}\left(  \bR^{d}\right)}.$
 
 \item (R2') We have to bound the quantity $ ||\cL_{t,n}-\cL_{t}||_{BV_{2}\to L^1}\leq \pi_n,$ where the quantity $\pi_n$ was defined at item R2 in section \ref{sec:perturbativereminder}.
 
 We have for $f\in BV_2$ and using Lemma \ref{lem:Rdregularization}: 

$$||(\mathcal{L}_{t,n}-\cL_{t})(f)||_{L^1}=
\int_ {B_n}|e^{is}-1||\mathcal{L}_tf|dx\le 2\text{Leb}(B_n)\|\mathcal{L}_tf\|_{C^1(\mathbb{R}^d)}\le 
$$
$$
2\text{Leb}(B_n)C_t \|f\|_{L^1(\mathbb{R}^d)}\le 2\text{Leb}(B_n)C_t \|f\|_{BV_2},
$$
   with $\pi_n=2\text{Leb}(B_n)C_t.$
 \item (R3') The closeness of the two operators is also quantified by 
$$
\Delta_n=\int (\mathcal{L}_t-\mathcal{L}_{t,n})f_0dx=(1-e^{is})\mu(B_n).
$$ 
We have now to show that 
$$\pi_n\|(\mathcal{L}_{t,n}-\mathcal{L}_t)(f_0)\|_{BV_2}\le \text{constant} \ |\Delta_n|.
$$
Since the density $f_0$ is locally bounded away from zero, see Lemma \ref{lem:Lemma3_posdensity}, it will be enough to show that the quantity $\|(e^{is1_{B_n}}-1)\mathcal{L}_t(f_0)\|_{BV_2}$ is bounded by a constant independent of $n.$ This follows by exactly the same arguments which allowed us to bound the $BV_2$ norm of $e^{is1_{B_n}}\mathcal{L}_t$ in item (R1').
   \item (R4) The quantities $q_k$\footnote{We use the same symbol as for the $q_k$ introduced in item R4 in section \ref{sec:perturbativereminder}}
    associated to this perturbation will therefore have the form
\begin{equation}\label{newq}
    q_k=\lim_{n\rightarrow \infty}\frac{1}{\Delta_n}\int (\mathcal{L}_t-\mathcal{L}_{t,n})\mathcal{L}^k_n(\mathcal{L}_t-\mathcal{L}_{t,n})(f_0)dx,
\end{equation}
provided the limits exists.
We now show that all these quantities are zero, just by repeating the proof of Proposition \ref{prop:R7}. Let us consider the function $g_{k,n}:=\frac{\mathcal{L}^k_{t,n}(\mathcal{L}_t-\mathcal{L}_{t,n})(f_0)}{(1-e^{is})\mu(B_n)}.$ We have
$$
q_k=\lim_{n\rightarrow \infty}\int (\mathcal{L}_t-\mathcal{L}_{t,n})g_{k,n}dx=\lim_{n\rightarrow \infty}\int_{B_n}(1-e^{is})\mathcal{L}_tg_{k,n} dx\le 
$$
$$
\lim_{n\rightarrow \infty}|1-e^{is}|\text{Leb}(B_n)\|\mathcal{L}_tg_{k,n}\|_{\infty}.
$$
By using lemma \ref{lem:Rdregularization}, we have that   $\|\mathcal{L}_tg_{k,n}\|_{\infty}\le C_t \|g_{k,n}\|_{L^1(\mathbb{R}^d)},$ and it is immediate to check that $\|g_{k,n}\|_{L^1(\mathbb{R}^d)}\le 1.$
 In conclusion the limit defining $q_k$ will tend to $0$ when $n\rightarrow \infty.$\\
   
  \end{itemize}

As in the proof of the Gumbel's law,  we need that, given the  number $\tau,$ the measure of the set $B_n$ scales like $n\mu(B_n)\rightarrow \tau, n\rightarrow \infty.$ 

Then we get for the top eigenvalue $\iota_n$ of $\mathcal{L}_{t,n},$ see Proposition \ref{thm:repfo}
$$
\iota_n=1-(1-\sum_{k=0}^{\infty}q_k)\Delta_n+o(\Delta_n)
=1-(1-e^{is})\mu(B_n)+o(\mu(B_n)).
$$
Since the assumptions of section \ref{sec:perturbativereminder} are verified by $\mathcal{L}_t$ and by  $\mathcal{L}_{t,n}$ thanks to the sufficient conditions quoted in   section \ref{sec:Rdperturbation}, we could repeat the steps from eq. (\ref{eq:perturbation}) to (\ref{eq:perturbation3}) showing that the leading term in the growing of $\Psi_n(s)$ is just given by the $n$-th power of $\iota_n.$
Therefore
$$
\lim_{n\rightarrow \infty}\Phi_n(s)=\lim_{n\rightarrow \infty}\int e^{isS_{n}}d\mathbb{P}d\mu=\lim_{n\rightarrow \infty}\iota_n^n=e^{- (1-e^{is}) \tau}:=\Phi(s).
$$

Notice that this is just the pointwise limit of the characteristic function of the random variable
$$
S_{n,\tau}:=\sum_{i=0}^{\lfloor\frac{\tau}{\mu(B_n)}\rfloor}1_{B_n}(X^x_{ih}).
$$
Since $\Phi(s)$ is  continuous in $s=0,$ it is the characteristic function of a random variable $W$ to which $S_{n,\tau}$ converges in distribution. But such a limiting variable $W$ has the Poisson distribution

$$
\nu_{W}(\{k\})=\frac{e^{-\tau}\tau^k}{k!}.
$$
\end{proof}

\appendix  
 
\section{Dependance from $h$ of the limit law}\label{A3}

Our main result, Theorem \ref{thm:gen2}, is formulated assuming the time discretization of step $h$ by considering $X_{nh}$. However,  note that the right hand side limiting law of the statement, equation \eqref{th2},  does not depend on $h$.   
Let us notice that the limit for $h \to 0$, of the above results does not provide more information: this does not come as a surprise as the passage from discrete to continuous usually requires a rescaling.

\begin{proposition}
For every $h>0$, let the rest of the datum be as in Theorem  \ref{thm:gen2}. It holds.
\[
\lim_{n\rightarrow\infty}\bP \otimes \mu \left( \left\{ (\omega,x) : \max_{t\in\left[  0,nh\right]
}g\left(  X_{t}^x \right)  \leq u_{n} \right\} \right)  =0.
\]
\end{proposition}
Note, with respect to Theorem \ref{thm:gen2} that here $ t \in [0,nh]$.  

\begin{proof}
Let $M>0$ be an integer.  In the rest of the  proof, we shorten $\bP \otimes \mu \left(\{ (\omega, x) : \ldots \} \right) := \bP \otimes \mu ( \ldots )$, characterizing the sets considered by the inequality that define it. 
We have
\[
\mathbb{P} \otimes \mu\left(  \max_{t\in\left[  0,nh\right]  }g\left(  X_{t}^x\right)  \leq
u_{n}\right)  \leq\mathbb{P} \otimes \mu\left(  \max_{k=0,...,nM}g\left(  X_{k\frac{h}{M}%
}^x\right)  \leq u_{n}\right)
\]
hence%
\[
\limsup_{n \to \infty} \mathbb{P} \otimes \mu\left(  \max_{t\in\left[  0,nh\right]  }g\left(
X_{t}^x\right)  \leq u_{n}\right)  \leq\lim_{n\rightarrow\infty}\mathbb{P} \otimes \mu%
\left(  \max_{k=0,...,nM}g\left(  X_{k\frac{h}{M}}^x\right)  \leq u_{n}\right)
.
\]
Now, if $n\mu_{0}\left(  B\left(  x_{0},e^{-u_{n}}\right)  \right)
\rightarrow\tau$, then
\[
nM\mu_{0}\left(  B\left(  x_{0},e^{-u_{n}}\right)  \right)  \rightarrow\tau M.
\]
Consider a sequence $\left\{  v_{j}\right\}  $ such that $v_{nM}=u_{n}$
for every $n$. We have%
\[
nM\mu_{0}\left(  B\left(  x_{0},e^{-v_{nM}}\right)  \right)  \rightarrow\tau
M.
\]
One can define $\left\{  v_{j}\right\}  $ in such a way that
\[
j\mu_{0}\left(  B\left(  x_{0},e^{-v_{j}}\right)  \right)  \rightarrow\tau M
\]
as $j\rightarrow\infty$, not only along the subsequence $j_{n}=nM$. Thus, by
the theorem (applied with $\frac{h}{M}$ in place of $h$)%
\[
\lim_{j\rightarrow\infty}\mathbb{P} \otimes \mu\left(  \max_{k=0,...,j}g\left(
X_{k\frac{h}{M}}\right)  \leq v_{j}\right)  =e^{-\tau M}%
\]
and thus, for the subsequence $j_{n}=nM$,
\[
\lim_{n\rightarrow\infty}\mathbb{P} \otimes \mu\left(  \max_{k=0,...,nM}g\left(
X_{k\frac{h}{M}}\right)  \leq v_{nM}\right)  =e^{-\tau M}.
\]
But $v_{nM}=u_{n}$. Hence%
\[
\lim\sup\mathbb{P} \otimes \mu \left(  \max_{t\in\left[  0,nh\right]  }g\left(
X_{t}\right)  \leq u_{n}\right)  \leq e^{-\tau M}.
\]
Since this holds true for every $M$, we complete the proof.
\end{proof}

The proposition indicates that in order  to extend the result to continuous time, one  might require a not obvious rescaling (see also on the matter \cite[Section 3.6]{1LR}). 

\section{Appendix} \label{A4}

An elementary result for independent identically distributed random variables
$\left(  X_{n}\right)  _{n}$ with law $\mu$ tells us that%
\[
\lim_{n\rightarrow\infty}\mathbb{P}\left(  X_{1}\notin B_{n},...,X_{n}\notin
B_{n}\right)  =e^{-\tau}%
\]
if $\left(  B_{n}\right)  _{n}$ is a sequence of Borel sets such that
$\lim_{n\rightarrow\infty}n\mu\left(  B_{n}\right)  =\tau$. In this appendix we
would like to mimic, in the framework of sequences of random variables, a
question arising in this paper for diffusions. The question has to do with the
intuition about our result for small $h$. Let us explain the issue. 

For large $h$, the SDE considered in this paper loses memory and thus the
sample $\left(  X_{nh}\right)  _{n}$ is made of r.v.'s which are roughly
independent. Thus the final result is not heuristically surprising, in view of
the previous remark on sequences of independent random variables.

But for small $h$, the sample $\left(  X_{nh}\right)  _{n}$ is highly
correlated. Why should we expect the same result? For instance, assume (this
is not the case for the SDE) that a sequence $\left(  Y_{n}\right)  _{n}$ of
r.v.'s is made of constant blocks of length $M$:%
\begin{align*}
Y_{1}  & =...=Y_{M}=X_{1}\\
& ...\\
Y_{\left(  n-1\right)  M+1}  & =...=Y_{nM}=X_{n}%
\end{align*}
with $\left(  X_{n}\right)  _{n}$ iid as above with law $\mu$. Then
\begin{align*}
& \mathbb{P}\left(  Y_{i}\notin B_{nM}\text{ for }i=1,2,...,nM\right)  \\
& =\mathbb{P}\left(  Y_{kM}\notin B_{nM},\text{ for }k=1,2,...,n\right)
=\left(  1-\mu\left(  B_{nM}\right)  \right)  ^{n}\\
& \sim\exp\left(  -n\mu\left(  B_{nM}\right)  \right)  \rightarrow
e^{-\frac{\tau}{M}}%
\end{align*}
instead of $e^{-\tau}$. 

The previous "counterexample" does not take into account the presence of
fluctuations at small time-scale, which exists in the case of the SDE. Let us
show the role of such fluctuations. Let $\left(  Y_{n}\right)  _{n}$ be so
defined: there is an integer $M>0$ such that
\begin{align*}
Y_{1} &  =X_{1}+W_{1},...,Y_{M}=X_{1}+W_{M}\\
Y_{M+1} &  =X_{2}+W_{M+1},...,Y_{2M}=X_{2}+W_{2M}\\
&  ...\\
Y_{\left(  k-1\right)  M+1} &  =X_{k}+W_{\left(  k-1\right)  M+1}%
,...,Y_{kM}=X_{k}+W_{kM}%
\end{align*}
where $\left(  W_{n}\right)  _{n}$ is a sequence of independent identically
distributed random variables with law $\nu$, independent of $\left(
X_{i}\right)  _{i}$. In other words,%
\begin{align*}
Y_{\left(  k-1\right)  M+i} &  =X_{k}+W_{\left(  k-1\right)  M+i}\\
\text{for }i &  =1,...,M\text{, for every }k=1,2,...
\end{align*}
Now the law of $Y_{n}$ is the convolution $\mu\ast\nu$. 

\begin{theorem}
Assume%
\begin{equation}
\lim_{n\rightarrow\infty}\sup_{x\in\mathbb{R}}\nu\left(  \,B_{n}+x\right)
=0\label{sufficiently diffuse}%
\end{equation}
and
\[
\lim_{n\rightarrow\infty}n\left(  \mu\ast\nu\right)  \left(  \,B_{n}\right)
=\tau.
\]
Then%
\[
\lim_{n\rightarrow\infty}\mathbb{P}\left(  Y_{i}\notin B_{nM}\text{ for
}i=1,2,...,nM\right)  =e^{-\,\tau}.
\]

\end{theorem}

\begin{proof}
We have%
\begin{align*}
& \mathbb{P}\left(  Y_{i}\notin B_{nM}\text{ for }i=1,...,nM\right)  \\
& =\mathbb{P}\left(  X_{1}+W_{j}\notin B_{nM}\text{ for }j=1,...,M\right)
^{n}.
\end{align*}
Now%
\begin{align*}
& \mathbb{P}\left(  X_{1}+W_{j}\notin B_{nM}\text{ for }j=1,...,M\right)  \\
& =\int P\left(  W_{j}\notin\,B_{nM}-x\text{ for }j=1,...,M\right)  \mu\left(
dx\right)  \\
& =\int\left(  1-\nu\left(  \,B_{nM}-x\right)  \right)  ^{M}\mu\left(
dx\right)  .
\end{align*}
Passing to logarithms, we have to prove that
\[
\lim_{n\rightarrow\infty}n\log\int\left(  1-\nu\left(  \,B_{nM}-x\right)
\right)  ^{M}\mu\left(  dx\right)  =-\,\tau.
\]
It is then sufficient to prove that
\[
\int\left(  1-\nu\left(  \,B_{nM}-x\right)  \right)  ^{M}\mu\left(  dx\right)
=1-\frac{\tau}{n}+o\left(  \frac{1}{n}\right)
\]
namely that
\begin{align*}
& \int M\nu\,\left(  B_{nM}-x\right)  \mu\left(  dx\right)  \\
& -\int\left(  \frac{M\left(  M-1\right)  }{2}\nu\left(  \,B_{nM}-x\right)
^{2}+...+\nu\left(  \,B_{nM}-x\right)  ^{M}\right)  \mu\left(  dx\right)  \\
& =\frac{\tau}{n}+o\left(  \frac{1}{n}\right)  .
\end{align*}
Since
\begin{align*}
\int M\nu\,\left(  B_{nM}-x\right)  \mu\left(  dx\right)    & =M\left(
\mu\ast\nu\right)  \left(  \,B_{nM}\right)  =\frac{1}{n}nM\left(  \mu\ast
\nu\right)  \left(  \,B_{nM}\right)  \\
& =\frac{\tau}{n}+o\left(  \frac{1}{n}\right)
\end{align*}
it remains to prove that
\[
\int\nu\left(  \,B_{nM}-x\right)  ^{k}\mu\left(  dx\right)  =o\left(  \frac
{1}{n}\right)
\]
for every integer $k\geq2$. One has%
\begin{align*}
\int\nu\left(  \,B_{nM}-x\right)  ^{k}\mu\left(  dx\right)    & \leq\sup
_{x}\nu\left(  \,B_{nM}-x\right)  \int\nu\left(  \,B_{nM}-x\right)  \mu\left(
dx\right)  \\
& =o\left(  1\right)  \left(  \frac{\tau}{nM}+o\left(  \frac{1}{n}\right)
\right)
\end{align*}
by the assumption and the previous computation, hence we have proved the
result. 
\end{proof}

The assumption (\ref{sufficiently diffuse}) on $\nu$, which we could call
"sufficiently diffuse", is not satisfied by a delta Dirac, which is the case
of the counter-example made before the theorem (it was the case $\nu=\delta
_{0}$, namely $W_{n}=0$ for every $n$). 

To summarize, the two ingredients that seem to have a fundamental importance
for our result are a loss of memory plus sufficiently rich local fluctuations.

\section*{Acknowledgement}
P.G acknowledges the support of the Centro di Ricerca Matematica Ennio de Giorgi and of UniCredit Bank R\&D group for financial support through the “Dynamics and Information Theory Institute” at Scuola Normale Superiore. P.G. was partially supported by INDAM - GNFM project "“Deterministic and stochastic dynamical systems for climate studies” and by PRIN 2022NTKXCX. The research of SV was supported by the project {\em Dynamics and Information Research Institute} within the agreement between UniCredit Bank and Scuola Normale Superiore di Pisa and by the Laboratoire International Associ\'e LIA LYSM, of the French CNRS and  INdAM (Italy). SV thanks the Mathematical Research Institute MATRIX,
the Sydney Mathematical Research Institute (SMRI), the University of New South Wales, and the University of Queensland for their support and hospitality and where 
 part of this research was performed. SV was also supported by the project MATHAmSud TOMCAT 22-Math-10, N. 49958WH, du french  CNRS and MEAE. 
 The research of S.G. was partially supported by  the
research project "Stochastic properties of dynamical systems" (PRIN 2022NTKXCX) of the Italian Ministry of Education and Research. S.G. also thanks Aix-Marseille Université for hospitality during the research.

\section*{Conflicts of interests, Competing interests.}

The authors have no conflict of interests or competing interests to declare that are relevant to the content of this article.

\printbibliography

\end{document}